\documentclass[a4paper,11pt]{article}

\usepackage[left=2cm,right=2cm,top=2.5cm,bottom=2.5cm]{geometry}
\usepackage[english]{babel}
\usepackage[utf8x]{inputenc}
\usepackage[T1]{fontenc}
\usepackage{amsmath}
\usepackage{amsfonts}
\usepackage{dsfont}
\usepackage{amsthm}
\usepackage{stmaryrd}
\usepackage{graphicx}
\usepackage{float}
\usepackage[font=footnotesize]{caption}
\usepackage[normalem]{ulem}
\usepackage{csquotes}
\usepackage{xcolor}
\usepackage{hyperref}
\usepackage{fancyref}
\usepackage{listings}
\usepackage{placeins}
\usepackage{float}

\usepackage{lmodern}

\usepackage{authblk}

\newenvironment{AMS}{\textbf{\textit{MSC 2020 Subject Classification:}}}{}
\newenvironment{keywords}{\textbf{\textit{Keywords:}}}{}
\newenvironment{acknowledgements}{\textbf{Acknowledgements:}}{}

\theoremstyle{plain}
\newtheorem{Thm}{Theorem}[section]
\newtheorem{Prop}[Thm]{Proposition}
\newtheorem{Lem}[Thm]{Lemma}
\newtheorem{rem}[Thm]{Remark}

\begin{document}
\title{A generalised spatial branching process with \emph{ancestral} branching to model the growth of a filamentous fungus}

\author{Lena Kuwata\footnote{lena.kuwata@u-paris.fr}}

\affil{Universit\'e Paris Cit\'e, CNRS, MAP5, 75006 Paris, France}

\maketitle

\begin{abstract}
In this work, we introduce a spatial branching process to model the growth of the mycelial network of a filamentous fungus. In this model, each filament is described by the position of its tip, the trajectory of which is solution to a stochastic differential equation with a drift term which depends on all the other trajectories. Each filament can branch either at its tip or along its length, that is to say at some past position of its tip, at some time- and space-dependent rates. It can stop growing at some rate which also depends on the positions of the other tips. We first construct the measure-valued process corresponding to this dynamics, then we study its large population limit and we characterise the limiting process as the weak solution to a system of partial differential equations.
\end{abstract}



\section{Introduction}

Filamentous fungi are able to grow in a vast array of environments, making them ubiquitous in nature. They colonise their environment by forming filamentous structures, called \emph{hyphae}, which lengthen and branch to create networks, called \emph{mycelia}, the scale of which can range from a few micrometres to a few kilometres. To feed the mycelium, the hyphae produce enzymes to decompose the surrounding organic matter, thereby playing a key role in the functioning of their ecosystem. The hyphae can also sense their local environment, and efficient chemical communication along the network then enables the mycelium to quickly respond to threats such as attacks by predators, physical obstacles, or noxious local conditions, by partially reorganising or reorienting the growth of the mycelium \cite{Boddy,Fricker}.

In this paper, we introduce a measure-valued stochastic process which aims at modelling the growth of a mycelium in interaction with itself and with its environment. For instance, we may be interested in studying the impact of the local harvesting and sharing of nutrients on the fungus growth, or in understanding the spatial and temporal scales of the communication across the network that is generated by the fungus being attacked by predators in some region of space. To this end, we need to incorporate the following key biological components of the structure and growth of a mycelium in the model. The network is composed of ``primary'' hyphae, which grow more or less radially, and of ``secondary'' hyphae which branch off from the primary structure, approximately uniformly along the existing hyphae (we shall later refer to it as \emph{lateral} branching). The primary hyphae expand the area covered by the mycelium, exploring the environment in search of new sources of nutrients. Their number increases by branching in two at their tips (also called \emph{apexes}; we shall later refer to this branching at the tip as \emph{apical} branching). The role of the secondary hyphae is to increase the density of the network by growing in different directions from the primary hyphae; these hyphae also branch at their tips and uniformly along the existing length. When two hyphae cross, they can sometimes merge (this is called \emph{anastomosis}), increasing the connectivity of the mycelium by creating shortcuts for the diffusion of molecules through the network. The hyphae are only a few micrometres in diameter, while the mycelium can cover up to a few square kilometres, demonstrating the multiscale nature of these networks. Although it is difficult to quantify such phenomena, hyphae are thought to interact with each other, for instance through a tendency to avoid each other in dense areas, but also through local sharing of different resources. Having a general model of interactions between hyphae will allow us to test which of these effects is stronger in experimental cultures of filamentous fungi \cite{Dikec,Ledoux-these} by confronting the observed patterns with the predictions derived from variants of the process based on different interaction kernels. This is one of the main objectives of the combination of theoretical and empirical research work initiated in \cite{Dikec} and pursued by the pluridisciplinary NEMATIC consortium, involving colleagues from biology, physics, mathematics, computer science and geography \cite{catellier24,Catellier-DAngelo-Ricci,chassereau2024full,ledoux2024characterization,ledoux2022prediction,ledoux2023prediction,Tomasevic-Bansaye-Veber}. As monitoring mycelial growth in the laboratory requires fungi to be constrained to grow in two dimensions (on the surface of a Petri dish), in our model hyphal tips are supposed to evolve in $\mathbb{R}^2$. In nature, fungi may also grow on surfaces; furthermore, the model can easily be formulated with a three-dimensional spatial structure.

The literature on mycelial growth modelling is already large. Some of the models focus on the scale of a single hypha and aim to understand the mechanisms triggering its extension (see, \emph{e.g.}, \cite{Balmant,Tindemans}). Other models focus on the scale of the whole mycelium, described by its total biomass, and use systems of ordinary differential equations to describe the interaction between the mycelium and its environment, and its effects on the growth of the fungus and on certain characteristics of the environment (see, \emph{e.g.}, \cite{Lamour}). One of the main difficulties lies in establishing a link between these two scales. To this end, a number of computational models have been developed. These models can be either lattice-based (for more computational efficiency) or lattice-free (yielding results that are closer to what is typically observed), and consist in modelling the growth and branching dynamics of the hyphae during one unit of time and using these rules to update the network at regular time steps. For a review, see \cite{Boswell-Davidson}. The analysis of these models is usually purely numerical, and relies on exploring the space of parameters and finding parameters producing the patterns that most closely match the observations, rather than deriving long-term growth properties of the mycelium. Spatially explicit models have been introduced based on reaction-diffusion partial differential equations \cite{Du-Tran-Perre}, or on a system of stochastic differential equations encoding the behaviour of each hypha and its mean-field deterministic limit~\cite{Catellier-DAngelo-Ricci}. The latter is similar to models of tumour-induced angiogenesis \cite{Flandoli} and allows to study global quantities such as the stationary shape and speed of the invasion front formed by the mycelium. In the model developed in \cite{Catellier-DAngelo-Ricci}, apical branching occurs at constant rate, lateral branching occurs at a rate proportional to the current length of the hypha, and anastomosis occurs at a constant rate. The motion of each apex is described by a Langevin-type second order stochastic differential equation driven by a field of nutrients which are absorbed along the hyphae. The evolution of the field of nutrients is described by a reaction-diffusion equation in which the reaction term corresponds to the absorption of the nutrients by the mycelium, so that the two systems of equations are coupled. Taking the mean-field limit of this model gives rise to a system of coupled partial differential equations for the evolution of the mycelium and of the field of nutrients. While the aim of our work is to, similarly, obtain a large population limit result to describe the behaviour of the network, we shall use a different approach to modelling the mycelium: we use two distinct mechanisms to describe apical and lateral branching events, we incorporate different sources of interactions between the filaments and we replace the Langevin-type movement of a single hypha by the solution to a stochastic differential equation with a Brownian component (with a constant variance coefficient that should be thought about as being small, despite our use of the generic notation $\sigma^2$), and a stronger drift term encoding potential interactions. In doing so, we complement the study proposed in~\cite{Catellier-DAngelo-Ricci} and explore alternative ways of modelling interactions and hyphal trajectories for comparison with experimental data.

As a first step towards using the data obtained thanks to the experimental setup described in \cite{Dikec} to infer certain growth parameters and quantify the impact of various forms of stress (nutrient depletion, pH, ...) on the mycelial growth and structure, a non-spatial growth-fragmentation model for the expansion of the mycelium was introduced in \cite{Tomasevic-Bansaye-Veber}. In this model, the interactions are left aside and it is assumed that over the timescale of an experiment, the consumption of nutrients by the mycelium is negligible, so that the environment remains homogeneous over time and the spatial organisation of the mycelium does not matter. The hyphae are encoded by their lengths and grow at a constant speed. Apical branching occurs at a constant rate and hyphae branch laterally at a rate proportional to their lengths, while anastomosis is neglected. An explicit expression for the stationary profile of the mean measure of the length of the hyphae is obtained, which, along with long time behaviour and law of large number results, gives an estimator for the empirical distribution of hyphal lengths in the long time. This allows to infer the growth parameters (elongation speed and branching rates) appearing in the model from a panorama of the state of the fungal mycelium at the end of a growth experiment \cite{inference}. Here, we aim at extending the model proposed in~\cite{Tomasevic-Bansaye-Veber} by taking into account the spatial structure of the mycelium and the interactions between the hyphae.

To this end, the mycelium is encoded as the history of a spatial birth-death process, which we keep calling a generalised branching process by analogy with the branching structure of the mycelium (although the branching property in a mathematical sense is obviously lost due to the interactions between the individuals). Each individual in our branching process corresponds to an apex and is encoded by its position $x \in \mathbb{R}^2$. Each hypha then corresponds to the trajectory followed by an individual from time $0$ in the process, and the mycelium corresponds to the set of all such trajectories. That is to say, at any time $t \ge 0$, the mycelial network is fully represented by the set
\begin{equation}\label{def mycelium}
\mathfrak{M}_t := \{X^u_s : u \in V_s, s\le t \},
\end{equation}
where $V_s$ is the index set of all individuals alive at time $s$, and $X^u_s$ is the position at time $s$ of the apex labelled by $u\in \mathbb{N}$. This ``historical'' representation of hyphae allows us to encode the time at which each infinitesimal portion of filament was created, which can prove useful to model interactions within the network and with the environment. For instance, if a portion of hypha created at position $X^u_s$ consumes nutrients at a constant rate $c$, then at time $t \ge s$, a concentration $c_0 e^{-c(t-s)}$ of nutrients will remain locally, where $c_0$ is the initial concentration present at position $X^u_s$. For mathematical convenience, the population of apexes alive at time $t$ will be represented by its empirical distribution rather than by a set as in \eqref{def mycelium} (see Section~\ref{Section 2}). We model the interactions within the network by taking the trajectories of the apexes to be solutions to stochastic differential equations with drift terms depending on all the other trajectories, similar to the ones coming from chemotaxis modelling (see, \emph{e.g.}, \cite{Jabir}). We shall assume that lateral and apical branching occurs at space-dependent rates $b_1$ and $b_2$ and that hyphae can stop growing at some rate $d$ which also depends on the position of the other apexes. The latter is a simplified way of modelling anastomosis (akin to a density-dependent regulation of local population sizes), in which the merging into another hypha is replaced by the disappearing of the apex where the death event happens (note that only the apex disappears from the system at its death time $\mathfrak{t}^u$, while its past trajectory $(X^u_s)_{0\leq s\leq \mathfrak{t}^u}$ keeps influencing the future of the process -- that is, the filament stays there but simply stops growing and branching apically). As the probability of an apex encountering another hypha is proportional to the occupation of space in the vicinity of the apex, the death rate should be taken proportional to the local density of pieces of filaments. However, to simplify the model, the death rate will depend only on the position of the apexes, instead of the whole history of the process, so that hyphae will stop growing at a rate proportional to the number of apexes present around their tips, which serves as a proxy for the occupation of space. In this model, the interactions with the environment are encoded by the space dependence of the branching and death rates $b_1$, $b_2$ and $d$, which will be deterministic. A possible extension of the model would therefore be to consider growth in a stochastic environment in which areas that are less conducive to growth are randomly distributed. This could be done by either locally increasing the death rate or locally decreasing the branching rates, as is done, for instance, in \cite{Veber} and \cite{Englander} respectively in the case of branching Brownian motion. The interactions between the mycelium and its environment could also involve interactions with individuals from other species, in a context of either competition or symbiosis, in the vein of catalytic branching processes (see \emph{e.g.} \cite{Dawson-Etheridge, Dawson-Perkins}).

Spatial birth-death processes have been used extensively to model the evolution of structured populations, starting with seminal work on branching Brownian motion, super-Brownian motion and Fleming-Viot superprocesses \cite{Dawson,dynkin,FlemingViot} (see also the reviews \cite{Etheridge,Perkins}), and going towards more ecologically driven models (see, \emph{e.g.} \cite{BansayeMeleard,BKE,Champagnat,FM}). In these models, each individual moves according to some Markov process and the interactions between individuals are modelled through a competition term in the individual death rate. Historical versions of these branching processes have also been considered, in which each individual alive at time $t$ is no longer represented by its current location in space, but by the whole path followed by its ancestors between time $0$ and $t$. Historical processes are classically encoded as measure-valued Markov processes \cite{DawsonPerkins,Tran} (the measures acting on the set of trajectories over finite time intervals), or as Markov processes with values in the set of measured metric spaces \cite{Depperschmidt,Greven09} (the individuals alive at time $t$ forming the leaves of an evolving tree inducing a distance between them). In \cite{Tran}, the measure-valued historical framework is used to allow the trajectories in space and the instantaneous reproduction and death rates of the individuals that are currently alive to depend on the characteristics of their ancestors.

The novelty of our model compared to these historical processes is the occurrence of birth events \emph{along the ancestry}. That is, not only extant individuals/apexes can give birth to new individuals appearing at time $t$ at the current location of the ``parent'', but a birth event can also occur at time~$t$ at the location $X_s^u$ of the ``ancestor'' at time $s$ of the individual labelled by $u$ at time $t$. From the modelling point of view, this corresponds to the fact that previously created pieces of filaments remain at all times and can be the place of a lateral branching event, happening several units of time after the passage of the apex. The tree describing the network structure is therefore not ultrametric. In addition, the trajectories of the apexes depend on the past of the entire mycelium, so that the equations describing these trajectories are coupled. The stochastic differential equations used to model the movement of the apexes are similar to those appearing in interacting particle systems associated to models for chemotaxis with past dependence, where each particle interacts with all the past of other particles by means of a time-integrated functional involving an interaction kernel, albeit with a bounded interaction kernel instead of the singular ones used for chemotaxis modelling (see, \emph{e.g.}, \cite{Budhiraja,Jabir}). However, note that in these models for chemotaxis, the number of individuals is held constant, while in our model births and deaths make the population size evolve in time. In particular, ancestral branching is another source of dependence on the past, since new individuals can appear at ancestral positions and influence the future motion of the other apexes.

The rest of the paper is organised as follows. In Section~\ref{Section 2}, we introduce our model, before stating our main results and providing a rigorous construction of our generalised branching process. In Section~\ref{Section 3}, we establish the martingale problem satisfied by the process. In Section~\ref{Section 4}, we take the large population limit of our model and show that the limiting process can be characterised as the solution to a system of partial differential equations. We also show that the limiting measure $\mathcal{Z}_t^\infty$ at any time $t>0$ is absolutely continuous with respect to Lebesgue measure on $\mathbb{R}^2$.

\section{Definition of the spatial birth-death process and statement of the main results}
\label{Section 2}

\subsection{Overview of the model and main results}\label{subs:overview}
The main results presented in this paper are $(i)$ the creation of an original and relevant model for the biological application we consider, $(ii)$ the construction of a well-defined non Markovian stochastic process with the desired dynamics, and finally $(iii)$ the derivation of a large population limit of this process which will open the door to a numerical analysis of its long term behaviour and a comparison with empirical data on fungal growth \cite{Dikec} in a future work.

\begin{rem}
From a mathematical point of view, the process we introduce here can be described as a spatial birth-death process with ``ancestral'' birth events and interactions between individuals. However, we shall use a terminology closer to the application that motivated this work. In particular, we shall talk about lateral branching when referring to a birth event occurring at the location of an ancestor, the ``regular'' reproduction of an individual will be called apical branching, and the movement of each individual will be described as the elongation of its past trajectory (corresponding to a hypha).
\end{rem}

The construction in itself uses a rather classical approach \cite{Champagnat,Fournier}, but the occurrence of lateral branching at a nonlinear rate requires good control of the number of individuals alive at any given time to show that there is no explosion of the total population size in finite time. Likewise, the proof of the large population limit will follow the usual steps (see also \cite{Fournier,Tran-these} and the numerous papers which followed the same line), but tightness of the sequence of processes indexed by $N$ (the order of magnitude of the initial population size, tending to infinity) is more delicate due to lateral branching, which creates new individuals along the history of the population at a rate proportional to the time during which the ancestral line has been existing. It will require a new argument to control the total size and spatial extent of the population over the time interval $[0,t]$, uniformly in $N$ (see Section~\ref{4.1.2}).

Let us first give an informal definition of the model and state the large population convergence result, which do not require the heavy notation that we shall use to rigorously construct the process of interest in Section~\ref{subs:construction}. At any time $t \ge 0$, each individual $u$ alive at time $t$ is encoded by its position $X^u_t \in \mathbb{R}^2$. The population is then represented by its empirical distribution
\begin{equation}\label{state at t}
 \mathcal{Z}_t = \sum_{u \in V_t} \delta_{X^u_t},
\end{equation}
where $V_t$ is the index set of all individuals alive at time $t$. We denote the set of all finite positive measures on $\mathbb{R}^2$ by $\mathcal{M}(\mathbb{R}^2)$, we endow it with the topology of weak convergence and we use the classical notation $\langle \nu,f\rangle:=\int f d\nu$ for the integral of the function $f$ with respect to the measure $\nu$ (when this quantity is well-defined).

The non-Markovian dynamics of the process $(\mathcal{Z}_t)_{t \ge 0}$ are as follows:
\begin{itemize}
    \item[(a)]\textbf{Elongation:} The trajectories of the individuals are solutions to stochastic differential equations with a drift term which depends on all the other trajectories. That is, the positions of the individuals are solutions to equations of the form
\begin{equation*}
\label{eq_X}
    dX^u_t = \sigma dW^u_t + \left[\int_0^t \int_{\mathbb{R}^2} L_{t-s}(X^u_t - x) \mathcal{Z}_s(dx)ds\right]dt,
\end{equation*}
where the $(W^u_t)_{t\ge 0}$ are independent standard Brownian motions and $L:\mathbb{R}_+ \times \mathbb{R}^2 \to \mathbb{R}^2$ is measurable and satisfies
\begin{align}
    & |L_t(x)| \le h_1(t), \ \forall t,x \in \mathbb{R}_+\times\mathbb{R}^2, \label{L_1}\\
    & |L_t(x)-L_t(y)| \le h_2(t) |x-y|, \ \forall t,x,y \in \mathbb{R}_+\times\mathbb{R}^2\times\mathbb{R}^2 \label{L_2},
\end{align}
where $h_1,h_2 : \mathbb{R}_+ \to \mathbb{R}_+$ are integrable, $h_1$ is bounded over any compact interval and $| \cdot |$ is the Euclidean norm on $\mathbb{R}^2$.
     \item[(b)] \textbf{Apical branching:} At any time $t \ge 0$, each individual $u$ alive at time $t$ branches at instantaneous rate $b_1(X^u_t)$, giving birth at time $t$ to a new individual initially located at the position $X^u_t$ of its parent. We suppose that $b_1:\mathbb{R}^2\rightarrow \mathbb{R}_+$ is continuous and bounded by a constant $B_1 > 0$.
    \item[(c)] \textbf{Lateral branching:} At any time $t \ge 0$, each position belonging to the ancestry of at least one individual $u$ (being therefore of the form $X^u_{s}$ with $s \in (0,t)$) branches at instantaneous rate $b_2(X^u_s)$, giving birth to a new individual initially located at $X^u_{s}$. We suppose that $b_2:\mathbb{R}^2\rightarrow \mathbb{R}_+$ is continuous and bounded by a constant $B_2 > 0$.
    \item[(d)] \textbf{Death:} At any time $t \ge 0$, each individual $u$ alive at time $t$ dies at instantaneous rate $d(X^u_t,\mathcal{Z}_t)$. We suppose that $d:\mathbb{R}^2\times \mathcal{M}(\mathbb{R}^2)\rightarrow \mathbb{R}_+$ is continuous and that there exist $C>0$ and a bounded measurable function $\Psi:\mathbb{R}^2\rightarrow \mathbb{R}$ such that for every $x \in \mathbb{R}^2$ and $\nu, \nu' \in \mathcal{M}(\mathbb{R}^2)$, we have
    \begin{align}
        &d(x,\nu) \le C \langle \nu , 1 \rangle, \qquad \qquad \qquad \qquad \qquad  \hbox{and} \label{d_1}\\
        &|d(x,\nu)-d(x,\nu')| \le  |\langle \nu-\nu',\Psi \rangle |. \label{d_2}
    \end{align}
\end{itemize}
\begin{rem}
\label{rem_assumptions}
\noindent $(i)$ Observe that the death rate of a single apex at time $t$ depends on the positions of the other apexes, and not on the full population history up to time~$t$. From the modelling point of view, it may be more relevant to assume dependence on the past (for instance to make the death rate depend on the local density of pieces of filaments, and not only on the local density of apexes), but the notation and the regularity condition imposed on $d$ would become more intricate. Therefore, we have chosen not to go in this direction for now, and to use the set of apexes as a proxy for the current occupation of space. By analogy with anastomosis, where two filaments merge when they cross, the death rate can be chosen of the form $d(x,\nu) = \langle \nu , f(x-.) \rangle$ for some appropriate bounded continuous function $f$, so that filaments stop growing at a rate proportional to the local occupation.

\noindent $(ii)$ Assumptions \eqref{L_1} and \eqref{L_2} and the integrability of $h_1,h_2$ ensure the well-posedness of the system of equations and have been taken from \cite[Chapter~2]{Tomasevic}. The additional assumption that $h_1$ is bounded over any compact interval is used in the coupling argument of the proof of tightness of the sequence of processes in the large population limit stated in Theorem~\ref{th:convergence}.

\noindent $(iii)$ The continuity of the branching and death rates and Assumption~\eqref{d_2} are not necessary for the construction of the process $(\mathcal{Z}_t)_{t\geq 0}$ and the derivation of its semi-martingale decomposition, and are only required for the large population limit. The continuity of $b_1$, $b_2$ and $d$ is needed for the characterisation of the limit and Assumption~\eqref{d_2} is needed in order to prove uniqueness of the limit.
\end{rem}

The process $(\mathcal{Z}_t)_{t\geq 0}$ will be constructed in a recursive way in Section~\ref{subs:construction}, but for now let us state the martingale problem it will satisfy, which will be at the basis of the large population limit derived just after. To this end, let $T>0$ and let $\mathcal{C}_b^{1,2}([0,T]\times \mathbb{R}^2)$ stand for the set of all bounded functions on $[0,T]\times\mathbb{R}^2$ of class $\mathcal{C}^1$ in time and $\mathcal{C}^2$ in space with bounded first and second order derivatives. In Section~\ref{Section 3}, we shall show the following result.
\begin{Prop}\label{prop:martingale pb}
For every $f \in \mathcal{C}_b^{1,2}([0,T]\times \mathbb{R}^2)$, the process
\begin{equation}
\begin{aligned}
\Bigg(\langle \mathcal{Z}_t,f_t \rangle - \langle \mathcal{Z}_0,f_0 \rangle - \int_0^t\int_{\mathbb{R}^2} \Bigg[\frac{\partial f_s}{\partial s}(x) & + \sum\limits_{i=1}^2 \left( \frac{\partial f_s}{\partial x_i}(x) \int_0^s\int_{\mathbb{R}^2} L^i_{s-r}(x-y) \mathcal{Z}_r(dy)dr + \frac{\sigma^2}{2} \frac{\partial^2 f_s}{\partial x_i^2}(x)\right)\\
    & + (b_1(x) + b_2(x)(t-s)-d(x,\mathcal{Z}_s))f_s(x) \Bigg] \mathcal{Z}_s(dx)ds\Bigg)_{0 \le t \le T} \label{martingale}
\end{aligned}
\end{equation}
is a martingale, where for $i\in \{1,2\}$, $L_{s-r}^i$ is the $i$-th coordinate of the function $L_{s-r}$.
\end{Prop}

In Section~\ref{Section 4}, we shall take the large population limit of the process $\mathcal{Z}$. Suppose that for every integer $N$, the process $\mathcal{Z}^{(N)}$ starts with a population of size of the order of $N$ and that the interaction kernel $L^N$ and death rate $d^N$ are of the form
\begin{equation}\label{normalised L d}
L^N_t(x) = \frac{1}{N}\, L_t(x) \quad \hbox{and}\quad  d^N(x,\nu) = d\Big(x,\frac{\nu}{N}\Big),
\end{equation}
where the functions $L$ and $d$ are independent of $N$ and satisfy the assumptions \eqref{L_1}, \eqref{L_2}, \eqref{d_1} and \eqref{d_2} stated above. Consider for every $N$ and every time $t\geq 0$ the renormalised measure
\begin{equation}\label{normalised Z}
\mathcal{Z}^N_t = \frac{1}{N} \,\mathcal{Z}^{(N)}_t.
\end{equation}
As the process $\mathcal{Z}^{(N)}$ satisfies the martingale problem stated in Proposition \ref{prop:martingale pb} with interaction kernel $L^N$ and death rate $d^N$, we obtain that for every $f \in \mathcal{C}_b^{1,2}([0,T]\times \mathbb{R}^2)$, the process
\begin{equation*}
\begin{aligned}
\Bigg(\langle \mathcal{Z}^N_t,f_t \rangle - \langle \mathcal{Z}^N_0,f_0 \rangle - \int_0^t\int_{\mathbb{R}^2} \Bigg[\frac{\partial f_s}{\partial s}(x) & + \sum\limits_{i=1}^2 \left( \frac{\partial f_s}{\partial x_i}(x) \int_0^s\int_{\mathbb{R}^2} L^i_{s-r}(x-y) \mathcal{Z}^N_r(dy)dr + \frac{\sigma^2}{2} \frac{\partial^2 f_s}{\partial x_i^2}(x)\right)\\
    & + (b_1(x) + b_2(x)(t-s)-d(x,\mathcal{Z}^N_s))f_s(x) \Bigg] \mathcal{Z}^N_s(dx)ds\Bigg)_{0 \le t \le T} 
\end{aligned}
\end{equation*}
is a martingale. Let $D_{\mathcal{M}(\mathbb{R}^2)}[0,\infty)$ be the set of all c\`adl\`ag paths with values in $\mathcal{M}(\mathbb{R}^2)$, which we endow with the standard Skorokhod topology. We shall show the following result.
\begin{Thm}\label{th:convergence}
Suppose that $(\mathcal{Z}^N)_{N\in \mathbb{N}}$ is constructed using an interaction kernel $L^{N}$ and a death rate $d^{N}$ as in \eqref{normalised L d}, where $L$ and $d$ satisfy \eqref{L_1}, \eqref{L_2}, \eqref{d_1} and \eqref{d_2} and that $(\mathcal{Z}_0^N)_{N\in \mathbb{N}}$ converges in law to a deterministic measure $\mathcal{Z}_0\in \mathcal{M}(\mathbb{R}^2)$ in such a way that $\sup_{N\geq 1} \mathbb{E}[\langle \mathcal{Z}_0^N,1 \rangle^2] < + \infty$. Then, the sequence of processes $(\mathcal{Z}^N)_{N \ge 1}$ converges in law in $D_{\mathcal{M}(\mathbb{R}^2)}[0,\infty)$ to a continuous deterministic process $\mathcal{Z}^\infty$ satisfying: $\mathcal{Z}_0^\infty=\mathcal{Z}_0$ and for every $f \in \mathcal{C}^2_b(\mathbb{R}^2)$ and $t \ge 0$,
\begin{equation}
\begin{aligned}
\langle \mathcal{Z}_t^\infty,f \rangle = \langle \mathcal{Z}_0^\infty,f \rangle + \int_0^t\int_{\mathbb{R}^2} &\Bigg[\sum\limits_{i=1}^2 \left(\frac{\partial f}{\partial x_i}(x) \int_0^s\int_{\mathbb{R}^2} L^i_{s-r}(x-y) \mathcal{Z}_r^\infty(dy)dr + \frac{\sigma^2}{2} \frac{\partial^2 f}{\partial x_i^2}(x)\right)\\ 
& \qquad + \big(b_1(x) + b_2(x)(t-s) -d(x,\mathcal{Z}_s^\infty)\big)f(x) \Bigg] \mathcal{Z}_s^\infty(dx)ds. \label{eq_Zlim}
\end{aligned}
\end{equation}
Furthermore, for every $t > 0$, $\mathcal{Z}_t^\infty$ is absolutely continuous with respect to Lebesgue measure on $\mathbb{R}^2$, and its density $\rho_t$ is a weak solution to: for every $t \ge 0$ 
\begin{equation}
\begin{aligned}
    &\frac{\partial \rho_t}{\partial t}(x) + \nabla_x \rho_t(x) \cdot \left( \int_0^t \int_{\mathbb{R}^2} L_{t-s}(x-y)\rho_s(y)dyds\right) + \rho_t(x) \int_0^t \int_{\mathbb{R}^2} \left(\nabla_x \cdot L_{t-s}(x-y)\right) \rho_s(y)dyds \\
   &= \frac{\sigma^2}{2} \Delta_x \rho_t(x) + (b_1(x)-d(x,\rho_t))\rho_t(x) + b_2(x)\int_0^t \rho_s(x)ds.
\end{aligned}
\end{equation}
\end{Thm}
The content of Theorem~\ref{th:convergence} is not very surprising in view of Proposition~\ref{prop:martingale pb} and our scaling assumptions on the different parameters. However, the occurrence of lateral branching, combined with the increase in the total length of the ancestry over a finite time interval as $N$ tends infinity (since the number of individuals in the population increases), generates new difficulties that we have to overcome to prove the results presented here (see Section~\ref{4.1.2}).

\subsection{Construction of the process}\label{subs:construction}

Let us now construct the process $(\mathcal{Z}_t)_{t \ge 0}$. To this end, we follow the algorithmic method introduced in \cite{Champagnat,Fournier}. Let $M$ be a Poisson random measure on $\mathbb{R}_+ \times \mathbb{N} \times \mathbb{R}_+ \times (0,1)$ with intensity measure $ds \otimes n(du) \otimes dz \otimes d\theta$, where $n(du)$ denotes the counting measure on $\mathbb{N}$. Let $((W^u_t)_{t \ge 0}, u \in \mathbb{N})$ be independent standard Brownian motions which are also independent of $M$ and let $\sigma>0$.

Let $\mathcal{Z}_0$ be a random finite counting measure such that $\mathbb{E}\left[ \langle \mathcal{Z}_0 , 1 \rangle ^2 \right] < +\infty$. We set $V_0 = \{1,\dots,I_0\}$, where $I_0$ denotes the number of atoms in $\mathcal{Z}_0$, and write $\mathcal{Z}_0 = \sum_{u \in V_0} \delta_{X^u_0}$. We define a first process $(\mathcal{Z}^0_t)_{t\geq 0}$ by $\mathcal{Z}^0_t = \mathcal{Z}_0$ for all $t\geq 0$.

Let $T_0 = 0$ and let $\{X^u_t, t \ge 0,  u \in V_0 \}$ satisfy the system of coupled equations
\begin{equation*}
dX^u_t = \sigma dW^u_t + \left[\sum\limits_{v \in V_0} \int_0^t L_{t-s}(X^u_t - X^v_s)ds\right]dt,
\end{equation*}
starting from $X^u_0$ for every $u\in V_0$. Under Assumption~\eqref{L_1}, the drift term in each equation is bounded by an integrable function, which ensures existence of a weak solution to the system (see Proposition~5.3.6 in~\cite{KS}). Using Assumption~\eqref{L_2}, we can show pathwise uniqueness of the solution (see Theorem~2.2.4 in~\cite{Tomasevic}). Weak existence combined with pathwise uniqueness then gives us strong existence and uniqueness of the solution to the system (see Corollary~1 in~\cite{Yamada-Watanabe}).

We now recursively construct the times $T_k$ of branching or death events, the sets $V_{T_k}$ of indices of individuals alive between times $T_{k}$ and $T_{k+1}$ and the processes $(\mathcal{Z}^k_t, (X^u_{T_k+t},t\ge 0,u\in V_{T_k}))$, where $(\mathcal{Z}^k_t)_{t \ge 0}$ is the process stopped at time $T_k$ and $(X^u_{T_k + t},t\ge 0,u\in V_{T_k})$ are the trajectories followed by the individuals alive at time $T_k$ in the absence of branching or death events.

Suppose that for some $k\ge 0$ we have constructed $T_k$, $V_{T_k}$ and $(\mathcal{Z}^k_t, (X^u_{T_k +t},t\ge 0,u\in V_{T_k}))$. Let
$$
\overline{\mathcal{Z}}^k_s = \displaystyle\sum\limits_{u \in V_{T_k}} \delta_{X^u_s},\qquad s\geq T_k,
$$
be the process after time $T_k$ in the absence of branching or death events (that is, individuals only move in space), and set
\begin{align*}
T_{k+1} = \inf  \Bigg\{t > T_k :  & \int_{(T_k,t]\times \mathbb{N}\times \mathbb{R}_+\times (0,1)}  \Big( \mathds{1}_{\{u \in V_{s-},a \le b_1(X^u_s)+d(X^u_s,\overline{\mathcal{Z}}^k_s) \}}\\
& \quad + \mathds{1}_{\{u \in V_{\theta s}, b_1(X^u_s)+d(X^u_s,\overline{\mathcal{Z}}^k_s) < a \le b_1(X^u_s)+b_2(X^u_{\theta s})s+d(X^u_s,\overline{\mathcal{Z}}^k_s) \}}\Big) M(ds,du,da,d\theta) > 0 \Bigg\}
\end{align*}
to be the first time at which a branching or death event occurs after time $T_k$.

\begin{rem}
To lighten the notation, the same Poisson point process is used for the apical branching, lateral branching and death mechanisms. To this end, instead of choosing $s$ uniformly in $(0,t)$ and having lateral branching occur at position $X_s$ at rate $b_2(X_s)$, as is suggested by the description of the model, we choose $\theta$ uniformly in $(0,1)$ and have lateral branching occur at position $X_{\theta t}$ at rate $b_2(X^u_{\theta t})t$. This alternative encoding gives rise to the same dynamics as the one informally described earlier.
\end{rem}

Let $(U_{k+1}, A_{k+1}, \theta_{k+1})$ be the unique triplet of variables such that $M(\{(T_{k+1},U_{k+1},A_{k+1},\theta_{k+1})\}) = 1$. If $A_{k+1} \le b_1(X^{U_{k+1}}_{T_{k+1}})$, individual $U_{k+1}$ branches apically, giving birth, at time $T_{k+1}$, to an individual labelled by $I_0 + k+1$ with initial position $X^{I_0 + k+1}_{T_{k+1}} = X^{U_{k+1}}_{T_{k+1}-}$. If $b_1(X^{U_{k+1}}_{T_{k+1}}) < A_{k+1} \le b_1(X^{U_{k+1}}_{T_{k+1}})+d(X^{U_{k+1}}_{T_{k+1}},\overline{\mathcal{Z}}^k_{T_{k+1}})$, individual $U_{k+1}$ dies at time $T_{k+1}$, that is, $U_{k+1}$ is removed from $V_{T_k}$. Finally, if $b_1(X^{U_{k+1}}_{T_{k+1}})+d(X^{U_{k+1}}_{T_{k+1}},\overline{\mathcal{Z}}^k_{T_{k+1}}) < A_{k+1} \le b_1(X^{U_{k+1}}_{T_{k+1}})+b_2(X^{U_{k+1}}_{\theta_{k+1}T_{k+1}}) T_{k+1}+d(X^{U_{k+1}}_{T_{k+1}},\overline{\mathcal{Z}}^k_{T_{k+1}})$, individual $U_{k+1}$ branches laterally, giving birth, at time $T_{k+1}$, to an individual with label $I_0 + k+1$ and initially located at $X^{I_0 + k+1}_{T_{k+1}} = X^{U_{k+1}}_{\theta_{k+1} T_{k+1}}$. Thus, we set
\begin{equation*}
    V_{T_{k+1}} =
    \begin{cases}
    V_{T_k} \cup \{I_0 + k + 1\}, \text{ if individual $U_{k+1}$ branches at time $T_{k+1}$},\\
    V_{T_k} \backslash \{U_{k+1}\}, \text{ if individual $U_{k+1}$ dies at time $T_{k+1}$}.
    \end{cases}
\end{equation*}
We then define the process $\mathcal{Z}^{k+1}$ as follows:
\begin{equation*}
    \mathcal{Z}^{k+1}_t =
    \begin{cases}
    \mathcal{Z}^{k}_t \text{ for } t < T_k,\\[2mm]
    \sum\limits_{u \in V_{T_k}} \delta_{X^u_t} \text{ for } T_k \le t < T_{k+1},\\[2mm]
    \sum\limits_{u \in V_{T_{k+1}}} \delta_{X^u_{T_{k+1}}} \text{ for } t \ge T_{k+1}.
    \end{cases}
\end{equation*}

As regards the motion of the (at most $I_0 + k+1$) individuals in $V_{T_{k+1}}$, let $(X^u_{T_{k+1}+t},t\ge 0,u\in V_{T_{k+1}})$ satisfy the system of coupled equations
\begin{equation*}
   \begin{aligned}
    dX^u_{T_{k+1}+t} = \sigma dW^u_{T_{k+1}+t} + \Bigg[&\int_0^{T_{k+1}} \int_{\mathbb{R}^2} L_{T_{k+1}+t-s}(X^u_{T_{k+1}+t}-x)\mathcal{Z}^{k+1}_s(dx)ds \\ & \quad + \sum\limits_{v \in V_{T_{k+1}}}\int_0^t L_{t-s}(X^u_{T_{k+1}+t}-X^v_{T_{k+1}+s})ds\Bigg] dt,
\end{aligned}
\end{equation*}
starting at $t = 0$ at the positions $X^u_{T_{k+1}}$. In the above equation, the drift term has been split into two terms, the first corresponding to interactions with the past of the process up to time $T_{k+1}$ and the second corresponding to the interactions between the individuals in $V_{T_{k+1}}$ from time $T_{k+1}$ on. This split is necessary as the process $\mathcal{Z}$ is not defined after time $T_{k+1}$ at this stage. Again, Assumptions~\eqref{L_1} and \eqref{L_2} ensure strong existence and uniqueness of the solution to this system.
This concludes the construction of the times $T_k$, sets $V_{T_k}$ and the processes $(\mathcal{Z}^k_t, (X^u_t,t\ge T_k,u\in V_{T_k}))$. To finish our construction, we now need to show the following result.

\begin{Prop}
\label{Prop_Tk}
For every $t \ge 0$, the number of branching and death events in the time interval $[0,t]$ is a.s. finite. That is to say, $\sup_{k \ge 0} \ T_k = +\infty$ a.s.
\end{Prop}
 
Indeed, once this property is established, we can define $\mathcal{Z}$ to be the process satisfying for all $k \in \mathbb{N}$ and $t \ge 0$,
\begin{equation*}
    \mathcal{Z}_{t \wedge T_k} = \mathcal{Z}^k_t,
\end{equation*}
and the fact that $\sup_{k \ge 0} \ T_k = +\infty$ a.s. ensures that the process is well defined at all times. We start by showing the following result.
\begin{Prop}
\label{Prop_eq_Z}
For every $f \in \mathcal{C}_b^2(\mathbb{R}^2)$, $k \ge 1$ and $t \in [0,T_{k})$, we have
\begin{align}
\langle \mathcal{Z}_t,f \rangle =& \langle \mathcal{Z}_0,f \rangle + \int_0^t\int_{\mathbb{R}^2} \left[\sum\limits_{i=1}^2 \left(\frac{\partial f}{\partial x_i}(x) \int_0^s\int_{\mathbb{R}^2} L^i_{s-r}(x-y) \mathcal{Z}_r(dy)dr + \frac{\sigma^2}{2} \frac{\partial^2 f}{\partial x_i^2}(x)\right) \right] \mathcal{Z}_s(dx)ds + M^f_{0,t} \nonumber\\
& +\int_{[0,t] \times \mathbb{N} \times \mathbb{R}_+ \times (0,1)} \Big(\left(\mathds{1}_{\{u\in V_{s-},a \le b_1(X^u_{s-})\}}-\mathds{1}_{\{u\in V_{s-},b_1(X^u_{s-}) < a \le b_1(X^u_{s-})+d(X^u_{s-},\mathcal{Z}_{s-})\}}\right) f(X^u_{s-})\nonumber \\
& \qquad +\mathds{1}_{\{u\in V_{\theta s},b_1(X^u_{s-})+d(X^u_{s-},\mathcal{Z}_{s-}) < a \le b_1(X^u_{s-})+b_2(X^u_{\theta s})s+d(X^u_{s-},\mathcal{Z}_{s-})\}} f(X^u_{\theta s}) \Big) M(ds,du,da,d\theta), \label{eq_Z}
\end{align}
where we have written $L = (L^1,L^2)$, and $M_{s,t}^f$ is defined for every $0\leq s\leq t$ by
\begin{equation}
\label{Mf}
M_{s,t}^f = \sum_{k \ge 1} \mathds{1}_{\{s \le T_{k-1} \le t\}} \sum_{u \in V_{T_{k-1}}} M_{s \vee T_{k-1},t\wedge T_k-}^{f,u},
\end{equation}
where for every $k \ge 1$, and $T_{k-1} \le s \le t < T_k$ 
\begin{equation}
\label{Mfu}
M_{s,t}^{f,u} =   f(X^u_t) - f(X^u_{s}) -  \int_{s}^t  \sum\limits_{i=1}^2 \left(\frac{\partial f}{\partial x_i}(X^u_r) \int_0^r \int_{\mathbb{R}^2}L^i_{r-r'}(X^u_r - y) \mathcal{Z}_{r'}(dy)dr' + \frac{\sigma^2}{2}\frac{\partial^2 f}{\partial x_i^2}(X^u_r)\right)dr.
\end{equation}
\end{Prop}

\begin{proof}[Proof of Proposition \ref{Prop_eq_Z}]
We show by induction on $k \ge 1$ that for every $t \in [T_{k-1},T_{k})$, $(\mathcal{Z}_t)_{t \ge 0}$ satisfies \eqref{eq_Z}. Let $k = 1$ and $t\in [T_0,T_{1})$. Since, by definition, no branching or death event occur during the interval of time $[0,t]$, we have
\begin{align*}
    &\int_{[0,t]\times \mathbb{N} \times \mathbb{R}_+ \times (0,1)} \Big( \left(\mathds{1}_{\{u\in V_{s-},a \le b_1(X^u_{s-})\}}-\mathds{1}_{\{u\in V_{s-},b_1(X^u_{s-}) < a \le b_1(X^u_{s-})+d(X^u_s,\mathcal{Z}_{s-})\}}\right) f(X^u_{s-})\\
     & \qquad + \mathds{1}_{\{u\in V_{\theta s},b_1(X^u_{s-})+d(X^u_s,\mathcal{Z}_{s-}) < a \le b_1(X^u_{s-})+b_2(X^u_{\theta s})s+d(X^u_{s-},\mathcal{Z}_{s-})\}} f(X^u_{\theta s}) \Big) M(ds,du,da,d\theta) = 0,
\end{align*}
and the result holds by definition of $M_{0,t}^f$.

Let $k \ge 1$, and suppose that for every $t \in [0,T_{k})$, $(\mathcal{Z}_t)_{t \ge 0}$ satisfies \eqref{eq_Z}. Let $t \in [T_{k}, T_{k+1})$ and set $V_{t,1} = V_{T_{k-1}} \backslash \{U_k\}$ and $V_{t,2} = V_{T_k} \backslash V_{t,1}$. That is, $V_{t,1}$ is the set of individuals alive at time $T_{k-1}$ that neither branch nor die at time $T_k$. If individual $U_k$ branches at time $T_k$, $V_{t,2}$ contains the indices of individual $U_k$ and its offspring, while if individual $U_k$ dies at time $T_k$, $V_{t,2}$ is empty. As we consider $t \in [T_{k}, T_{k+1})$, the sets $V_{t,1}$ and $V_{t,2}$ depend on $k$, but this dependence has been suppressed for ease of notation. Note that $V_t = V_{T_k}$, hence $V_t =V_{t,1} \cup V_{t,2}$ and
\begin{equation*}
    \langle \mathcal{Z}_t,f \rangle = \sum\limits_{u \in V_{t,1}} f(X^u_t) + \sum\limits_{u \in V_{t,2}} f(X^u_t).
\end{equation*}
By definition of $M^{f,u}_{T_k,t}$, we can write
\begin{align*}
    \sum\limits_{u \in V_{t,1}} f(X^u_t) = \sum\limits_{u \in V_{t,1}} \Bigg[ f(X^u_{T_{k}}) + M^{f,u}_{T_{k},t} + \int_{T_{k}}^t  \sum\limits_{i=1}^2 \Bigg(\frac{\partial f}{\partial x_i}(X^u_s)  \int_0^s \int_{\mathbb{R}^2} & L^i_{s-r}(X^u_s - y) \mathcal{Z}_r(dy)dr \\
     & \quad + \frac{\sigma^2}{2}\frac{\partial^2 f}{\partial x_i^2}(X^u_s)\Bigg)ds \Bigg].
\end{align*}
Since $V_{t,1} = V_{T_{k-1}} \backslash \{U_k\}$, we have
\begin{equation*}
    \sum\limits_{u \in V_{t,1}} f(X^u_{T_{k}}) = \left(\sum\limits_{u \in V_{T_{k-1}}} f(X^u_{T_{k}-})\right) - f(X^{U_k}_{T_{k}-}).
\end{equation*}
Similarly, we can write
\begin{align*}
    \sum\limits_{u \in V_{T_{k-1}}} f(X^u_{T_{k}-}) = \sum\limits_{u \in V_{T_{k-1}}}  & \Bigg[ f(X^u_{T_{k-1}}) + M^{f,u}_{T_{k-1},T_k-} \\
     & + \int_{T_{k-1}}^{T_k}  \sum\limits_{i=1}^2 \Bigg(\frac{\partial f}{\partial x_i}(X^u_s) \int_0^s \int_{\mathbb{R}^2}L^i_{s-r}(X^u_s - y) \mathcal{Z}_r(dy)dr + \frac{\sigma^2}{2}\frac{\partial^2 f}{\partial x_i^2}(X^u_s)\Bigg)ds \Bigg].
\end{align*}
We then obtain
\begin{align}
& \sum\limits_{u \in V_{t,1}} f(X^u_t) \nonumber\\& = \langle \mathcal{Z}_{T_{k-1}},f \rangle - f(X^{U_k}_{T_{k}-}) \nonumber\\
        & \quad + \sum\limits_{u \in V_{T_{k-1}}} \left[ M^{f,u}_{T_{k-1},T_k-} + \int_{T_{k-1}}^{T_k}  \sum\limits_{i=1}^2 \left(\frac{\partial f}{\partial x_i}(X^u_s) \int_0^s \int_{\mathbb{R}^2}L^i_{s-r}(X^u_s - y) \mathcal{Z}_r(dy)dr + \frac{\sigma^2}{2}\frac{\partial^2 f}{\partial x_i^2}(X^u_s)\right)ds \right] \nonumber\\
        & \quad + \sum\limits_{u \in V_{t,1}} \left[ M^{f,u}_{T_{k},t} + \int_{T_{k}}^t  \sum\limits_{i=1}^2 \left(\frac{\partial f}{\partial x_i}(X^u_s) \int_0^s \int_{\mathbb{R}^2}L^i_{s-r}(X^u_s - y) \mathcal{Z}_r(dy)dr + \frac{\sigma^2}{2}\frac{\partial^2 f}{\partial x_i^2}(X^u_s)\right)ds \right].\label{eq_Z_1}
\end{align}
Again, we can write
\begin{align*}
    \sum\limits_{u \in V_{t,2}} f(X^u_t) = \sum\limits_{u \in V_{t,2}} \bigg[ & f\big(X^u_{T_{k}}\big) + M^{f,u}_{T_{k},t} \\
    & + \int_{T_{k}}^t \sum\limits_{i=1}^2 \left(\frac{\partial f}{\partial x_i}(X^u_s) \int_0^s \int_{\mathbb{R}^2}L^i_{s-r}(X^u_s - y) \mathcal{Z}_r(dy)dr + \frac{\sigma^2}{2}\frac{\partial^2 f}{\partial x_i^2}(X^u_s)\right)ds \bigg].
\end{align*}
If individual $U_k$ branches apically at time $T_k$, then $V_{T_k} = V_{T_{k-1}} \cup \{I_0 + k\}$, $V_{t,2} = \{U_{k},I_0 + k\}$ and \[\displaystyle\sum\limits_{u \in V_{t,2}} f\left(X^u_{T_k}\right) = 2f\left(X^{U_{k}}_{T_k-}\right).\]
If individual $U_k$ dies at time $T_k$, then $V_{T_{k}} = V_{T_{k-1}} \backslash \{U_{k}\}$, $V_{t,2} = \emptyset$ and \[\displaystyle\sum\limits_{u \in V_{t,2}} f\left(X^u_{T_k}\right) = 0.\]
If individual $U_k$ branches laterally at time $T_k$, then $V_{T_k} = V_{T_{k-1}} \cup \{I_0 + k\}$, $V_{t,2} = \{U_{k},I_0 + k\}$ and \[\displaystyle\sum\limits_{u \in V_{t,2}} f\left(X^u_{T_k}\right) = f\left(X^{U_{k}}_{T_k-}\right)+f\left(X^{U_{k}}_{\theta_k T_k}\right).\]
We have $t \in [T_k,T_{k+1})$ so that the only branching or death event occurring during the interval $(T_{k-1},t]$ is at time $T_k$, and we can write
\begin{align*}
    \sum\limits_{u \in V_{t,2}} f(X^u_{T_k}) = &\int_{(T_{k-1},t]\times \mathbb{N}\times\mathbb{R}_+\times (0,1)} \Big(\mathds{1}_{\{u \in V_{s-}, a \le b_1(X^u_{s-})\}} 2f(X^u_{s-}) \\
    & \qquad +\mathds{1}_{\{u \in V_{\theta s}, b_1(X^u_{s-})+d(X^u_{s-},\mathcal{Z}_{s-}) < a \le b_1(X^u_{s-})+b_2(X^u_{\theta s})s+d(X^u_{s-},\mathcal{Z}_{s-})\}}(f(X^u_{\theta s})+f(X^u_{s-}))\Big) \\
    & \hspace{10.5cm}  \times M(ds,du,da,d\theta).
\end{align*}
Similarly,
\begin{align*}
    f\left(X^{U_k}_{T_{k}-}\right) = &\int_{(T_{k-1},t]\times \mathbb{N}\times\mathbb{R}_+\times (0,1)} \Big(\mathds{1}_{\{u \in V_{s-}, a \le b_1(X^u_{s-})+d(X^u_{s-},\mathcal{Z}_{s-})\}} \\
    & \qquad+ \mathds{1}_{\{u \in V_{\theta s}, b_1(X^u_{s-})+d(X^u_{s-},\mathcal{Z}_{s-}) < a \le b_1(X^u_{s-})+b_2(X^u_{\theta s})s+d(X^u_{s-},\mathcal{Z}_{s-})\}} \Big)f(X^u_{s-}) \\
    & \hspace{8.7cm} \times M(ds,du,da,d\theta).
\end{align*}
We then obtain
\begin{align}
&\sum\limits_{u \in V_{t,2}} f\left(X^u_t\right) \nonumber\\
& = f\left(X^{U_k}_{T_{k}-}\right) +\sum \limits_{u \in V_{t,2}} \left[ M^{f,u}_{T_{k},t} + \int_{T_{k}}^t  \sum\limits_{i=1}^2 \left(\frac{\partial f}{\partial x_i}(X^u_s) \int_0^s \int_{\mathbb{R}^2}L^i_{s-r}(X^u_s - y) \mathcal{Z}_r(dy)drds + \frac{\sigma^2}{2}\frac{\partial^2 f}{\partial x_i^2}(X^u_s)ds\right) \right]\nonumber\\
&\quad +\int_{(T_{k-1},t]\times \mathbb{N}\times\mathbb{R}_+\times (0,1)} \Big((\mathds{1}_{\{u \in V_{s-}, a \le b_1(X^u_{s-})\}}-\mathds{1}_{\{u \in V_{s-}, b_1(X^u_{s-}) < a \le b_1(X^u_{s-})+d(X^u_{s-},\mathcal{Z}_{s-})\}}) f(X^u_{s-}) \label{eq_Z_2}  \\
& \qquad \qquad+ \mathds{1}_{\{u \in V_{\theta s}, b_1(X^u_{s-})+d(X^u_{s-},\mathcal{Z}_{s-}) < a \le b_1(X^u_{s-})+b_2(X^u_{\theta s})s+d(X^u_{s-},\mathcal{Z}_{s-})\}}f(X^u_{\theta s})\Big) M(ds,du,da,d\theta).\nonumber 
\end{align}
Adding \eqref{eq_Z_1} and \eqref{eq_Z_2}, we obtain
\begin{align*}
    &\langle \mathcal{Z}_t,f \rangle \\
    &= \langle \mathcal{Z}_{T_{k-1}}, f \rangle \\
    & \quad + \sum\limits_{u \in V_{T_k}} \left[ M^{f,u}_{T_{k},t} + \int_{T_{k}}^t  \sum\limits_{i=1}^2 \left(\frac{\partial f}{\partial x_i}(X^u_s) \int_0^s \int_{\mathbb{R}^2}L^i_{s-r}(X^u_s - y) \mathcal{Z}_r(dy)drds + \frac{\sigma^2}{2}\frac{\partial^2 f}{\partial x_i^2}(X^u_s)ds\right) \right]\\
    &\quad  + \sum\limits_{u \in V_{T_{k-1}}} \left[ M^{f,u}_{T_{k-1},T_k} + \int_{T_{k-1}}^{T_k}  \sum\limits_{i=1}^2 \left(\frac{\partial f}{\partial x_i}(X^u_s) \int_0^s \int_{\mathbb{R}^2}L^i_{s-r}(X^u_s - y) \mathcal{Z}_r(dy)drds + \frac{\sigma^2}{2}\frac{\partial^2 f}{\partial x_i^2}(X^u_s)ds\right) \right]\\
    &\quad  +\int_{(T_{k-1},t]\times \mathbb{N}\times\mathbb{R}_+\times (0,1)} \Big((\mathds{1}_{\{u \in V_{s-}, a \le b_1(X^u_{s-})\}}-\mathds{1}_{\{u \in V_{s-}, b_1(X^u_{s-}) < a \le b_1(X^u_{s-})+d(X^u_{s-},\mathcal{Z}_{s-})\}}) f(X^u_{s-}) \\
    & \qquad \qquad + \mathds{1}_{\{u \in V_{\theta s}, b_1(X^u_{s-})+d(X^u_{s-},\mathcal{Z}_{s-}) < a \le b_1(X^u_{s-})+b_2(X^u_{\theta s})s+d(X^u_{s-},\mathcal{Z}_{s-})\}}f(X^u_{\theta s})\Big) M(ds,du,da,d\theta).
\end{align*}
Noticing that $V_s = V_{T_{k-1}}$ for every $s \in [T_{k-1},T_k)$ and $V_s = V_{T_k}$ for every $s \in [T_k,t)$, we can write
\begin{align}
\langle \mathcal{Z}_t,f \rangle = & \langle \mathcal{Z}_{T_{k-1}}, f \rangle + \int_{T_{k-1}}^t \int_{\mathbb{R}^2} \left[\sum\limits_{i=1}^2 \left(\frac{\partial f}{\partial x_i}(x) \int_0^s\int_{\mathbb{R}^2} L^i_{s-r}(x-y) \mathcal{Z}_r(dy)dr + \frac{\sigma^2}{2} \frac{\partial^2 f}{\partial x_i^2}(x)\right) \right] \mathcal{Z}_s(dx)ds \nonumber\\
& +\int_{(T_{k-1},t]\times \mathbb{N}\times\mathbb{R}_+\times (0,1)} \Big((\mathds{1}_{\{u \in V_{s-}, a \le b_1(X^u_{s-})\}}-\mathds{1}_{\{u \in V_{s-}, b_1(X^u_{s-}) < a \le b_1(X^u_{s-})+d(X^u_{s-},\mathcal{Z}_{s-})\}}) f(X^u_{s-}) \nonumber\\
& \qquad \quad + \mathds{1}_{\{u \in V_{\theta s}, b_1(X^u_{s-})+d(X^u_{s-},\mathcal{Z}_{s-}) < a \le b_1(X^u_{s-})+b_2(X^u_{\theta s})s+d(X^u_{s-},\mathcal{Z}_{s-})\}}f(X^u_{\theta s})\Big) M(ds,du,da,d\theta) \nonumber\\
& + \sum\limits_{u \in V_{T_{k-1}}} M^{f,u}_{T_{k-1},T_k}+ \sum\limits_{u \in V_{T_k}} M^{f,u}_{T_{k},t}.\label{eq_Z_3}
\end{align}
Now, using the induction hypothesis we have
\begin{equation}
\begin{aligned}
& \langle \mathcal{Z}_{T_{k-1}},f \rangle \\
& = \langle \mathcal{Z}_0,f \rangle + \int_0^{T_{k-1}}\int_{\mathbb{R}^2} \left[\sum\limits_{i=1}^2 \left(\frac{\partial f}{\partial x_i} \int_0^s\int_{\mathbb{R}^2} L^i_{s-r}(x-y) \mathcal{Z}_r(dy)dr + \frac{\sigma^2}{2} \frac{\partial^2 f}{\partial x_i^2}\right) \right] \mathcal{Z}_s(dx)ds +  M^f_{0,T_{k-1}} \\
& \quad + \int_{[0,T_{k-1}] \times \mathbb{N} \times \mathbb{R}_+\times (0,1)} \Big((\mathds{1}_{\{u \in V_{s-}, a \le b_1(X^u_{s-})\}}-\mathds{1}_{\{u \in V_{s-}, b_1(X^u_{s-}) < a \le b_1(X^u_{s-})+d(X^u_{s-},\mathcal{Z}_{s-})\}}) f(X^u_{s-}) \\
& \hspace{1.5cm}+ \mathds{1}_{\{u \in V_{\theta s}, b_1(X^u_{s-})+d(X^u_{s-},\mathcal{Z}_{s-}) < a \le b_1(X^u_{s-})+b_2(X^u_{\theta s})s+d(X^u_{s-},\mathcal{Z}_{s-})\}}f(X^u_{\theta s})\Big) M(ds,du,da,d\theta).\label{eq_Z_4}
\end{aligned}
\end{equation}
Therefore, combining \eqref{eq_Z_3} and \eqref{eq_Z_4}, we have for every $t\in [T_k,T_{k+1})$,
\begin{align*}
    \langle \mathcal{Z}_t,f \rangle =& \langle \mathcal{Z}_0,f \rangle + \int_0^t\int_{\mathbb{R}^2} \left[\sum\limits_{i=1}^2 \left(\frac{\partial f}{\partial x_i} \int_0^s\int_{\mathbb{R}^2} L^i_{s-r}(x-y) \mathcal{Z}_r(dy)dr + \frac{\sigma^2}{2} \frac{\partial^2 f}{\partial x_i^2}\right) \right] \mathcal{Z}_s(dx)ds+M^f_{0,t}\\
    &+ \int_{[0,t] \times \mathbb{N} \times \mathbb{R}_+\times (0,1)} \Big((\mathds{1}_{\{u \in V_{s-}, a \le b_1(X^u_{s-})\}}-\mathds{1}_{\{u \in V_{s-}, b_1(X^u_{s-}) < a \le b_1(X^u_{s-})+d(X^u_{s-},\mathcal{Z}_{s-})\}}) f(X^u_{s-}) \\
    & \qquad + \mathds{1}_{\{u \in V_{\theta s}, b_1(X^u_{s-})+d(X^u_{s-},\mathcal{Z}_{s-}) < a \le b_1(X^u_{s-})+b_2(X^u_{\theta s})s+d(X^u_{s-},\mathcal{Z}_{s-})\}}f(X^u_{\theta s})\Big) M(ds,du,da,d\theta),
\end{align*}
which concludes the proof of Proposition \ref{Prop_eq_Z}.
\end{proof}

We can now show that $\sup_{k \ge 0} \ T_k = +\infty$ a.s., that is for every $t \ge 0$, the number of branching and death events occurring in the time interval $[0,t]$ is a.s. finite. The technique is standard and follows the same line as the proof of the corresponding result in \cite{Champagnat}.
\begin{proof}[Proof of Proposition~\ref{Prop_Tk}]
For every $j\in \mathbb{N}$, let $\tau_j = \inf \{t > 0 : \langle \mathcal{Z}_t,1 \rangle = j \}$, with the convention that $\tau_j=0$ if $\langle\mathcal{Z}_0,1\rangle\geq j$. Taking $f \equiv 1$ in \eqref{eq_Z}, we can write for every $t\geq 0$:
\begin{align*}
    \langle \mathcal{Z}_{t\wedge \tau_j},1 \rangle & = \langle \mathcal{Z}_0,1 \rangle +  \int_{[0,t\wedge \tau_j] \times \mathbb{N} \times \mathbb{R}_+\times (0,1)} \Big(\mathds{1}_{\{u\in V_{s-},a \le b_1(X^u_{s-})\}}-\mathds{1}_{\{u\in V_{s-},b_1(X^u_{s-}) < a \le b_1(X^u_{s-})+d(X^u_{s-},\mathcal{Z}_{s-})\}}\\
    & \qquad \qquad \qquad+ \mathds{1}_{\{u\in V_{\theta s},b_1(X^u_{s-})+d(X^u_{s-},\mathcal{Z}_{s-}) < a \le b_1(X^u_{s-})+b_2(X^u_{\theta s})s+d(X^u_{s-},\mathcal{Z}_{s-})\}}\Big) M(ds,du,da,d\theta)\\
    &\le \langle \mathcal{Z}_0,1 \rangle + \int_{[0,t\wedge \tau_j] \times \mathbb{N} \times \mathbb{R}_+ \times (0,1)} \Big(\mathds{1}_{\{u\in V_{s-},a \le b_1(X^u_{s-})\}}\\
    & \qquad \qquad \qquad + \mathds{1}_{\{u\in V_{\theta s},b_1(X^u_{s-})+d(X^u_{s-},\mathcal{Z}_{s-}) < a \le b_1(X^u_{s-})+b_2(X^u_{\theta s})s+d(X^u_{s-},\mathcal{Z}_{s-})\}}\Big) M(ds,du,da,d\theta).
\end{align*}
Thus, taking expectations on both sides of the inequality and using the fact that $b_1$ and $b_2$ are bounded by assumption, we obtain
\begin{align}
\mathbb{E}\left[\langle \mathcal{Z}_{t\wedge \tau_j},1 \rangle\right] \nonumber & \le \ \mathbb{E}\left[\langle \mathcal{Z}_0,1 \rangle\right] + \mathbb{E}\Bigg[\int_{0}^{t\wedge \tau_j} \int_{\mathbb{N} \times \mathbb{R}_+\times (0,1)} \Big(\mathds{1}_{\{u\in V_{s-},a \le b_1(X^u_{s-})\}}\nonumber \\
& \qquad \qquad \quad  +\mathds{1}_{\{u\in V_{\theta s},b_1(X^u_{s-})+d(X^u_{s-},\mathcal{Z}_{s-}) < a \le b_1(X^u_{s-})+b_2(X^u_{\theta s})s+d(X^u_{s-},\mathcal{Z}_{s-})\}}\Big) n(du)\,da\,d\theta\,ds\Bigg]\nonumber\\
& \le \ \mathbb{E}\left[\langle \mathcal{Z}_0,1 \rangle\right] + \mathbb{E}\left[\int_{0}^{t\wedge \tau_j} \int_0^1 \left(\sum\limits_{u \in V_{s-}} b_1(X^u_{s-}) + \sum\limits_{u \in V_{\theta s}} b_2(X^u_{\theta s})s \right)d\theta \,ds \right]\nonumber\\
& \le \ \mathbb{E}\left[\langle \mathcal{Z}_0,1 \rangle\right] + B_1 \ \int_{0}^{t}  \mathbb{E}\left[\langle \mathcal{Z}_{s\wedge \tau_j},1 \rangle\right] ds +  B_2\,\mathbb{E}\left[\int_0^{t\wedge \tau_j} \int_0^1 s\langle \mathcal{Z}_{\theta s},1 \rangle d\theta ds\right]. \label{expectation}
\end{align}
Let us consider the last term on the r.h.s. of \eqref{expectation}. Changing variables in $r = \theta s$, we can write
\begin{align*}
    B_2\, \mathbb{E}\left[\int_0^{t\wedge \tau_j} \int_0^1 s \ \langle \mathcal{Z}_{\theta s},1 \rangle\right] d\theta ds
    & = B_2\, \mathbb{E}\left[\int_0^{t\wedge \tau_j} \int_0^s \langle \mathcal{Z}_{r},1 \rangle\right] dr ds\\
    &= B_2 \,\mathbb{E}\left[ \int_0^{t\wedge \tau_j} (t\wedge \tau_j -r) \, \langle \mathcal{Z}_{r},1 \rangle dr \right]\\
    & \leq B_2 \,\mathbb{E}\left[\int_0^{t\wedge \tau_j} (t-r) \, \langle \mathcal{Z}_{r\wedge \tau_j},1 \rangle dr \right]\\
    & \leq B_2 \int_0^{t} (t-r) \, \,\mathbb{E}\big[\langle \mathcal{Z}_{r\wedge \tau_j},1 \rangle\big] dr.
\end{align*}
Coming back to \eqref{expectation}, we obtain
\begin{equation*}
    \mathbb{E}\left[\langle \mathcal{Z}_{t\wedge \tau_j},1 \rangle\right] \le \mathbb{E}\left[\langle \mathcal{Z}_0,1 \rangle\right] + \int_0^t (B_1 + (t-s)B_2) \ \mathbb{E}\left[\langle \mathcal{Z}_{s\wedge \tau_j},1 \rangle\right] ds,
\end{equation*}
and this holds true for every $t\geq 0$. Gronwall's lemma (see \emph{e.g.} Theorem 16 in \cite{Gronwall}) then yields for each given integer $j$,
\begin{equation}\label{gronwall}
\mathbb{E}\left[\langle \mathcal{Z}_{t\wedge\tau_j},1 \rangle\right] \le \mathbb{E}\left[\langle \mathcal{Z}_0,1 \rangle\right] e^{B_1 t + \frac{B_2}{2}t^2}, \qquad \forall t\geq 0.
\end{equation}
Let now $t>0$ and set $\alpha_t = \mathbb{P}(\sup_{n\in \mathbb{N}} \tau_n \leq t)$. We have for every large $j$
\begin{equation*}
\mathbb{E}\left[\langle \mathcal{Z}_{t\wedge\tau_j},1 \rangle\right] \ge \mathbb{E}\left[\langle \mathcal{Z}_{t\wedge\tau_j},1 \rangle \mathds{1}_{\{\sup_{n\in \mathbb{N}} \tau_n \leq t\}} \right] =  j \alpha_t.
\end{equation*}
If $\alpha_t$ were positive, then letting $j$ tend to infinity would yield a contradiction with \eqref{gronwall}. Therefore, we can conclude that $\alpha_t =0$, and since this holds true for any $t>0$, we obtain that
$$
\sup_{n\in \mathbb{N}} \tau_n =+\infty \qquad \hbox{a.s.}
$$
Using Fatou's Lemma, we can also conclude that
\begin{equation}
\label{ineq_Z}
    \mathbb{E}\left[\langle \mathcal{Z}_t,1 \rangle\right] \le \mathbb{E}\left[\langle \mathcal{Z}_0,1 \rangle\right] e^{B_1t + \frac{B_2}{2} t^2}, \qquad \forall t\geq 0.
\end{equation}

Recall that our aim is to show that the sequence $(T_k)_{k\geq 1}$ of jump times of our process tends to infinity almost surely as $k\rightarrow \infty$. To this end, for every $t\geq 0$ let
\begin{align*}
N_t = & \int_{[0,t] \times \mathbb{N} \times \mathbb{R}_+\times (0,1)} \Big(\mathds{1}_{\{u\in V_{s-},a \le b_1(X^u_{s-})\}} \\
& \qquad +\mathds{1}_{\{u\in V_{\theta s},b_1(X^u_{s-}) + d(X^u_{s-},\mathcal{Z}_{s-}) < a \le b_1(X^u_{s-})+b_2(X^u_{\theta s})s+d(X^u_{s-},\mathcal{Z}_{s-})\}}\Big) M(ds,du,da,d\theta)
\end{align*}
be the number of branching events occurring during the interval of time $[0,t]$. Using the same reasoning as before, we have
\begin{align}
\label{ineq_N}
    \mathbb{E}[N_t] &\le \int_0^t (B_1 + (t-s)B_2) \ \mathbb{E}\left[\langle \mathcal{Z}_s,1 \rangle\right] ds \nonumber \\
    & \le \mathbb{E}\left[\langle \mathcal{Z}_0,1 \rangle\right] e^{B_1t + \frac{B_2}{2} t^2} \ \int_0^t (B_1 + (t-s)B_2)ds   \nonumber \\
    & \le \mathbb{E}[\langle \mathcal{Z}_0,1 \rangle] \left(B_1t + \frac{B_2}{2}t^2\right)e^{B_1t + \frac{B_2}{2} t^2}.
\end{align}
Likewise, let $D_t$ be the number of deaths occurring during the interval of time $[0,t]$. Since the number of deaths $D_t$ cannot exceed the number of individuals that have lived during $[0,t]$, we have $D_t \le \langle \mathcal{Z}_0,1 \rangle + N_t$ and so
\begin{equation*}
    \mathbb{E}[N_t + D_t] \le \mathbb{E}[\langle \mathcal{Z}_0,1 \rangle] \left(1 + (2B_1t + B_2 t^2)e^{B_1t + \frac{B_2}{2} t^2}\right) < \infty.
\end{equation*}
This implies that for every $t \ge 0$, the number of branching and death events occurring during $[0,t]$ is a.s. finite, and so $\mathbb{P}(\sup_{k \to \infty} T_k \le t) = 0$, which concludes the construction of the process. 
\end{proof}

\section{Martingale problem}
\label{Section 3}

In this section, we prove Proposition~\ref{prop:martingale pb} for $f \in \mathcal{C}^2_b(\mathbb{R}^2)$. The proof is quite standard and relies on controlling the number of branching and death events and the total branching and death rates over any time interval $[0,t]$. The control on the number of branching and death events has already been obtained in the proof of Proposition~\ref{Prop_Tk}. The assumption that $b_1$ and $b_2$ are bounded by $B_1$ and $B_2$ respectively, combined with \eqref{ineq_Z}, gives us a bound on the expectation of the total branching rates. Finally, as $d(x,\mathcal{Z}_t) \le C \langle \mathcal{Z}_t, 1\rangle$, the expectation of the total death rate is bounded by the second moment of the total mass of the process, which we shall bound using the following result, whose proof is postponed until the end of the section for the sake of clarity.
 \begin{Lem}
    \label{Lem_ineq_Z2}
    For every $t \ge 0$
    \begin{equation}
\label{ineq_Z2}
    \mathbb{E}\left[\sup\limits_{0 \le s \le t}\langle \mathcal{Z}_s,1 \rangle^2\right] \le \mathbb{E}\left[\langle \mathcal{Z}_0,1 \rangle^2\right] e^{3\left(B_1t + \frac{B_2}{2}t^2\right)}.
\end{equation}
\end{Lem}

\begin{rem}Following the same reasoning as in the proof of Proposition~\ref{Prop_eq_Z}, an equation analogous to \eqref{eq_Z} can be obtained for $f \in \mathcal{C}^{1,2}_b([0,T] \times \mathbb{R}^2)$, and the proof of Proposition~\ref{prop:martingale pb} in the general case then follows the same reasoning as in the following proof.\end{rem}

\begin{proof}[Proof of Proposition~\ref{prop:martingale pb}]
Let $f \in \mathcal{C}^2_b(\mathbb{R}^2)$. For every $t \ge 0$, we have
\begin{align*}
    & \int_{[0,t] \times \mathbb{N} \times \mathbb{R}_+ \times (0,1)} \Big(\left(\mathds{1}_{\{u\in V_{s-},a \le b_1(X^u_{s-})\}}-\mathds{1}_{\{u\in V_{s-},b_1(X^u_{s-}) < a \le b_1(X^u_{s-})+d(X^u_{s-},\mathcal{Z}_{s-})\}}\right) f(X^u_{s-})\\
    & \hspace{1cm}+ \mathds{1}_{\{u\in V_{\theta s},b_1(X^u_{s-})+d(X^u_{s-},\mathcal{Z}_{s-}) < a \le b_1(X^u_{s-})+b_2(X^u_{\theta s})s+d(X^u_{s-},\mathcal{Z}_{s-})\}} f(X^u_{\theta s}) \Big) n(du)\,da\,d\theta\,ds\\
    & = \int_0^t \left(\sum\limits_{u \in V_{s-}} \left(b_1(X^u_{s-}) - d(X^u_{s-},\mathcal{Z}_{s-})\right)f(X^u_{s-}) + \int_0^1 \left(\sum\limits_{u \in V_{\theta s}}b_2(X^u_{\theta s})s f(X^u_{\theta s})\right)d\theta \right)ds \\
    & = \int_0^t \left(\int_{\mathbb{R}^2} \left(b_1(x)-d(x,\mathcal{Z}_s)\right)f(x) \mathcal{Z}_s(dx)  +\int_0^1\int_{\mathbb{R}^2} b_2(x)sf(x)  \mathcal{Z}_{\theta s}(dx) d\theta\right) ds.\\
\end{align*}
Changing variables in $r = \theta s$, we can write
\begin{align*}
    \int_0^t \int_0^1 \int_{\mathbb{R}^2} b_2(x)sf(x) \mathcal{Z}_{\theta s}(dx) d\theta ds &= \int_0^t \int_0^s \int_{\mathbb{R}^2} b_2(x)f(x)\mathcal{Z}_r(dx) dr ds\\
    &= \int_0^t (t-r) \int_{\mathbb{R}^2} b_2(x)f(x)\mathcal{Z}_r(dx) dr.
\end{align*}
Hence,
\begin{equation}
\begin{aligned}
    & \int_{[0,t] \times \mathbb{N} \times \mathbb{R}_+ \times (0,1)} \Big(\left(\mathds{1}_{\{u\in V_{s-},a \le b_1(X^u_{s-})\}}-\mathds{1}_{\{u\in V_{s-},b_1(X^u_{s-}) < a \le b_1(X^u_{s-})+d(X^u_{s-},\mathcal{Z}_{s-})\}}\right) f(X^u_{s-})\\
    & \qquad+ \mathds{1}_{\{u\in V_{\theta s},b_1(X^u_{s-})+d(X^u_{s-},\mathcal{Z}_{s-}) < a \le b_1(X^u_{s-})+b_2(X^u_{\theta s})s+d(X^u_{s-},\mathcal{Z}_{s-})\}} f(X^u_{\theta s}) \Big) n(du)\,da\,d\theta\,ds,\\
    & = \int_0^t\int_{\mathbb{R}^2}  \left(b_1(x)+b_2(x)(t-s)-d(x,\mathcal{Z}_s)\right)f(x) \mathcal{Z}_s(dx)ds. \label{eq_3.1}
\end{aligned}
\end{equation}
Combining \eqref{eq_Z} and \eqref{eq_3.1}, we can then write for every $t \ge 0$
\begin{equation}
\label{eq_Z2}
\begin{aligned}
    &\langle \mathcal{Z}_t,f \rangle - \langle \mathcal{Z}_0,f \rangle - \int_0^t\int_{\mathbb{R}^2} \Bigg[\sum\limits_{i=1}^2 \left(\frac{\partial f}{\partial x_i}(x) \int_0^s\int_{\mathbb{R}^2} L^i_{s-r}(x-y) \mathcal{Z}_r(dy)dr + \frac{\sigma^2}{2} \frac{\partial^2 f}{\partial x_i^2}(x)\right) \\
    & \hspace{5.5cm}+ (b_1(x) + b_2(x)(t-s)-d(x,\mathcal{Z}_s))f(x) \Bigg] \mathcal{Z}_s(dx)ds\\
    &= M_{0,t}^f + \int_{[0,t] \times \mathbb{N} \times \mathbb{R}_+ \times (0,1)} \Big(\left(\mathds{1}_{\{u\in V_{s-},a \le b_1(X^u_{s-})\}}-\mathds{1}_{\{u\in V_{s-},b_1(X^u_{s-}) < a \le b_1(X^u_{s-})+d(X^u_{s-},\mathcal{Z}_{s-})\}}\right) f(X^u_{s-})\\
    & \hspace{4.5cm} + \mathds{1}_{\{u\in V_{\theta s},b_1(X^u_{s-})+d(X^u_{s-},\mathcal{Z}_{s-}) < a \le b_1(X^u_{s-})+b_2(X^u_{\theta s})s+d(X^u_{s-},\mathcal{Z}_{s-})\}} f(X^u_{\theta s}) \Big)\\
    & \hspace{9cm}\times(M(ds,du,da,d\theta)-ds\,n(du)\,da\,d\theta),
\end{aligned}
\end{equation}
where $M^f_{0,t}$ was defined in Proposition \ref{Prop_eq_Z}.

We start by showing that $(M_{0,t}^f)_{t \ge 0}$ is a martingale. Let $k \ge 1$, $u \in V_{T_{k-1}}$, and $t \in [T_{k-1}, T_{k})$. Using Ito's formula, we can write
\begin{align*}
    f(X^u_t) =& f\left(X^u_{T_{k-1}}\right) + \sigma \sum\limits_{i=1}^2 \int_{T_{k-1}}^t \frac{\partial f}{\partial x_i}(X^u_s) dW^{u,i}_s \\
    & + \sum\limits_{i=1}^2 \int_{T_{k-1}}^t  \left[\frac{\partial f}{\partial x_i}(X^u_s) \int_0^s \int_{\mathbb{R}^2}L^i_{s-r}(X^u_s - x) \mathcal{Z}_r(dx)dr+ \frac{\sigma^2}{2}\frac{\partial^2 f}{\partial x_i^2}(X^u_s)\right]ds,
\end{align*}
where we have written $W^u = (W^{u,1},W^{u,2})$ for the two coordinates of the standard Brownian motion corresponding to individual $u$. Thus, recalling the definitions of $M^{f,u}$ and $M^{f}$ given in Proposition~\ref{Prop_eq_Z}, we have
\begin{equation*}
    M^{f,u}_{T_k,t} = \sigma \sum\limits_{i=1}^2 \int_{T_k}^t \frac{\partial f}{\partial x_i}(X^u_s) dW^{u,i}_s,
\end{equation*}
and
\begin{equation*}
    M^f_{s,t} = \sigma \sum_{u \in \mathbb{N}} \sum_{i = 1}^2 \int_s^t \mathds{1}_{\{ u \in V_r \}} \frac{\partial f}{\partial x_i}\left(X_r^u\right) dW_r^{u,i}.
\end{equation*}
Consequently, for any $s \ge 0$, $(M_{s,t}^f)_{t \ge 0}$ is a local martingale with quadratic variation
\begin{equation}
\label{ineq_M}
    \left\langle M^f_{s,.} \right\rangle_t = \sigma^2 \sum_{u \in \mathbb{N}} \sum_{i = 1}^2 \int_s^t \mathds{1}_{\{ u \in V_r \}} \left(\frac{\partial f}{\partial x_i}\left(X_r^u\right)\right)^2 dr \le \sigma^2 \|f\|_{2,\infty}^2 (t-s)(\langle \mathcal{Z}_s,1 \rangle + N_{s,t}),
\end{equation}
where
\begin{equation}\label{norm}
\|f\|_{2,\infty} = \displaystyle\sum\limits_{|\alpha| \le 2} \|\partial^{\alpha}f\|_\infty,
\end{equation}
and $N_{s,t}$ is the number of branching events occurring during $[s,t]$. Taking expectations on both sides of \eqref{ineq_M} and using \eqref{ineq_Z} and \eqref{ineq_N}, we obtain
\begin{equation*}
    \mathbb{E}\left[\left\langle M^f_{0,.} \right\rangle_t\right] \le \sigma^2 \|f\|^2_{2,\infty} t \, \mathbb{E}\left[\langle \mathcal{Z}_0 ,1 \rangle\right] \left(1 + B_1t + \frac{B_2}{2}t^2\right)e^{B_1t + \frac{B_2}{2} t^2}  < \infty
\end{equation*}
for every $t \ge 0$, which suffices to prove that $(M_{0,t}^f)_{t \ge 0}$ is a martingale.

We now show that
\begin{align}
    &\Bigg(\int_{[0,t] \times \mathbb{N} \times \mathbb{R}_+ \times (0,1)} \Big(\left(\mathds{1}_{\{u\in V_{s-},a \le b_1(X^u_{s-})\}}-\mathds{1}_{\{u\in V_{s-},b_1(X^u_{s-}) \le a \le b_1(X^u_{s-})+d(X^u_{s-},\mathcal{Z}_{s-})\}}\right) f(X^u_{s-}) \label{H}\\
    &+ \mathds{1}_{\{u\in V_{\theta s},b_1(X^u_{s-})+d(X^u_{s-}) < a \le b_1(X^u_{s-})+b_2(X^u_{\theta s})s+d(X^u_{s-},\mathcal{Z}_{s-})\}} f(X^u_{\theta s}) \Big)(M(ds,du,da,d\theta)-dsn(du)dzd\theta)\Bigg)_{t\ge 0}\nonumber
\end{align}
is a martingale.
Let
\begin{align*}
    H(s,u,a,\theta) =&\left(\mathds{1}_{\{u\in V_{s-},a \le b_1(X^u_{s-})\}}-\mathds{1}_{\{u\in V_{s-},b_1(X^u_{s-}) < a \le b_1(X^u_{s-})+d(X^u_{s-},\mathcal{Z}_{s-})\}}\right) f(X^u_{s-})\\
    & \quad + \mathds{1}_{\{u\in V_{\theta s},b_1(X^u_{s-})+d(X^u_{s-},\mathcal{Z}_{s-}) < a \le b_1(X^u_{s-})+b_2(X^u_{\theta s})s+d(X^u_{s-},\mathcal{Z}_{s-})\}} f(X^u_{\theta s}).
\end{align*}
Since $\theta \in (0,1)$, the process $(H(s,u,a,\theta))_{s \ge 0}$ is predictable. Thus, to show that the process in \eqref{H} is a martingale, it suffices to show that for every $t \ge 0$,
\begin{equation}\label{integrability}
    \mathbb{E}\left[\displaystyle\int_0^t \int_{\mathbb{N}\times \mathbb{R_+}\times (0,1)}|H(s,u,a,\theta)| ds\, n(du)\, da\, d\theta\right] < +\infty.
\end{equation}
For every $t \ge 0$, we have
\begin{align*}
    &\int_{0}^t \int_{\mathbb{N} \times \mathbb{R}_+\times (0,1)} \Big|\Big(\mathds{1}_{\{u\in V_{s-},z \le b_1(X^u_{s-})\}}-\mathds{1}_{\{u\in V_{s-},b_1(X^u_{s-}) < z \le b_1(X^u_{s-})+d(X^u_{s-},\mathcal{Z}_{s-})\}}\Big) f(X^u_{s-})\\
    &\hspace{2.5cm}+ \mathds{1}_{\{u\in V_{\theta s},b_1(X^u_{s-})+d(X^u_{s-}) < z \le b_1(X^u_{s-})+b_2(X^u_{\theta s})s+d(X^u_{s-},\mathcal{Z}_{s-})\}} f(X^u_{\theta s})\Big|ds\,n(du)\,dz\,d\theta \\
    & \le \|f\|_{2,\infty} \int_{0}^t \int_{\mathbb{N} \times \mathbb{R}_+\times (0,1)} \Big(\mathds{1}_{\{u\in V_{s-},z \le b_1(X^u_{s-})+d(X^u_{s-},\mathcal{Z}_{s-})\}}\\
    &\hspace{3cm} + \mathds{1}_{\{u\in V_{\theta s},b_1(X^u_{s-})+d(X^u_{s-},\mathcal{Z}_{s-}) < z \le b_1(X^u_{s-})+b_2(X^u_{\theta s})s+d(X^u_{s-},\mathcal{Z}_{s-})\}}\Big)ds\,n(du)\,dz\,d\theta\\
    & \le \|f\|_{2,\infty} \int_{0}^t \left[ B_1 \langle \mathcal{Z}_s,1\rangle + \sum\limits_{u \in V_{s}}d(X^u_{s}, \mathcal{Z}_{s}) +  B_2s \int_0^1 \langle \mathcal{Z}_{\theta s},1\rangle d\theta \right]ds.
\end{align*}
Recalling that $d(X^u_{s}, \mathcal{Z}_{s}) \le C \langle \mathcal{Z}_s,1\rangle$, we then have
\begin{align*}
    &\int_{0}^t \int_{\mathbb{N} \times \mathbb{R}_+\times (0,1)} \Big|\Big(\mathds{1}_{\{u\in V_{s-},z \le b_1(X^u_{s-})\}}-\mathds{1}_{\{u\in V_{s-},b_1(X^u_{s-}) < z \le b_1(X^u_{s-})+d(X^u_{s-},\mathcal{Z}_{s-})\}}\Big) f(X^u_{s-})\\
    & \hspace{2.5cm} + \mathds{1}_{\{u\in V_{\theta s},b_1(X^u_{s-})+d(X^u_{s-},\mathcal{Z}_{s-}) < z \le b_1(X^u_{s-})+b_2(X^u_{\theta s})s+d(X^u_{s-},\mathcal{Z}_{s-})\}} f(X^u_{\theta s})\Big|ds\,n(du)\,dz\,d\theta \\
    & \le \|f\|_{2,\infty} \int_0^t \left[(B_1+C \langle \mathcal{Z}_s,1\rangle)\langle \mathcal{Z}_s,1\rangle +  B_2s \int_0^1 \langle \mathcal{Z}_{\theta s},1\rangle d\theta \right]ds.
\end{align*}
As earlier, we can write, using Fubini's theorem,
\begin{align*}
    \int_0^t  \int_0^1 B_2s \langle \mathcal{Z}_{\theta s},1\rangle d\theta ds &= B_2 \int_0^t \int_0^s  \langle \mathcal{Z}_{r},1\rangle dr ds = B_2 \int_0^t (t-r)  \langle \mathcal{Z}_{r},1\rangle dr.
\end{align*}
Thus,
\begin{align}
    &\int_{0}^t \int_{\mathbb{N} \times \mathbb{R}_+\times (0,1)} \Big|\Big(\mathds{1}_{\{u\in V_{s-},z \le b_1(X^u_{s-})\}}-\mathds{1}_{\{u\in V_{s-},b_1(X^u_{s-}) < z \le b_1(X^u_{s-})+d(X^u_{s-},\mathcal{Z}_{s-})\}}\Big) f(X^u_{s-}) \nonumber\\
    & \hspace{2.5cm} + \mathds{1}_{\{u\in V_{\theta s},b_1(X^u_{s-})+d(X^u_{s-},\mathcal{Z}_{s-}) < z \le b_1(X^u_{s-})+b_2(X^u_{\theta s})s+d(X^u_{s-},\mathcal{Z}_{s-})\}} f(X^u_{\theta s})\Big|ds\,n(du)\,dz\,d\theta \nonumber \\
    & \le \|f\|_{2,\infty} \int_0^t \left[(B_1 + B_2(t-s))\langle \mathcal{Z}_s,1\rangle +  C \langle \mathcal{Z}_s,1\rangle^2\right] ds.\label{ineq_H}
\end{align}
To bound the expectation of the quantity on the right hand side of \eqref{ineq_H}, we use the result of Lemma~\ref{Lem_ineq_Z2}. Taking expectations on both sides of \eqref{ineq_H} and using \eqref{ineq_Z} and \eqref{ineq_Z2}, we then obtain
\begin{align*}
    &\mathbb{E}\Bigg[\int_{0}^t \int_{\mathbb{N} \times \mathbb{R}_+\times (0,1)} \Big|\Big(\mathds{1}_{\{u\in V_{s-},z \le b_1(X^u_{s-})\}}-\mathds{1}_{\{u\in V_{s-},b_1(X^u_{s-}) < z \le b_1(X^u_{s-})+d(X^u_{s-},\mathcal{Z}_{s-})\}}\Big) f(X^u_{s-})\\
    & \hspace{3cm}+ \mathds{1}_{\{u\in V_{\theta s},b_1(X^u_{s-})+d(X^u_{s-},\mathcal{Z}_{s-}) < z \le b_1(X^u_{s-})+b_2(X^u_{\theta s})s+d(X^u_{s-},\mathcal{Z}_{s-})\}} f(X^u_{\theta s})\Big|ds\,n(du)\,dz\,d\theta \Bigg] \\
    &\le C \|f\|_{2,\infty} \mathbb{E}\left[\langle \mathcal{Z}_0,1 \rangle^2\right]  e^{3\left(B_1t + \frac{B_2}{2}t^2\right)} < \infty,
\end{align*}
thus proving that the process in \eqref{H} is a martingale.
We have obtained that for every $f \in \mathcal{C}_b^2(\mathbb{R}^2)$, the process
\begin{align*}
    \Bigg( \langle \mathcal{Z}_t,f \rangle - \langle \mathcal{Z}_0,f \rangle - \int_0^t\int_{\mathbb{R}^2} \Bigg[ & \sum\limits_{i=1}^2 \left(\frac{\partial f}{\partial x_i}(x) \int_0^s\int_{\mathbb{R}^2} L^i_{s-r}(x-y) \mathcal{Z}_r(dy)dr + \frac{\sigma^2}{2} \frac{\partial^2 f}{\partial x_i^2}(x)\right) \\
    & \ + (b_1(x) + b_2(x)(t-s)-d(x,\mathcal{Z}_s))f(x) \Bigg] \mathcal{Z}_s(dx)ds\Bigg)_{t\ge 0}
\end{align*}
is a martingale, and the proof of Proposition \ref{prop:martingale pb} is complete.
\end{proof}

\begin{proof} [Proof of Lemma \ref{Lem_ineq_Z2}]
We can write for every $s \ge 0$,
\begin{align*}
    \langle \mathcal{Z}_s,1 \rangle^2 =& \langle \mathcal{Z}_0,1 \rangle^2 + \int_{[0,s] \times \mathbb{N}\times \mathbb{R}_+\times (0,1)} \Big( \mathds{1}_{\{u \in V_{r-}, a \le b_1(X^u_{r-})\}}\left[(\langle\mathcal{Z}_{r-},1 \rangle + 1)^2-\langle\mathcal{Z}_{r-},1 \rangle^2\right] \\
    & \quad+ \mathds{1}_{\{u \in V_{\theta r},b_1(X^u_{r-})+d(X^u_{r-},\mathcal{Z}_{r-}) < a \le b_1(X^u_{r-}) +b_2(X^u_{\theta r})r + d(X^u_{r-}\mathcal{Z}_{r-})\}}\left[(\langle\mathcal{Z}_{r-},1 \rangle + 1)^2-\langle\mathcal{Z}_{r-},1 \rangle^2\right]\\
    & \quad + \mathds{1}_{\{u \in V_{r-},b_1(X^u_{r-}) < a \le b_1(X^u_{r-})+d(X^u_{r-},\mathcal{Z}_{r-})\}} \left[(\langle\mathcal{Z}_{r-},1 \rangle - 1)^2-\langle\mathcal{Z}_{r-},1 \rangle^2\right]\Big)M(dr,du,da,d\theta)\\
    \le & \langle \mathcal{Z}_0,1 \rangle^2 + \int_{[0,s]\times \mathbb{N}\times \mathbb{R}_+\times (0,1)} \left(2 \langle\mathcal{Z}_{r-},1 \rangle + 1 \right) \Big(\mathds{1}_{\{u \in V_{r-},a \le b_1(X^u_{s-})\}}\\
    & \quad + \mathds{1}_{\{u \in V_{\theta r},b_1(X^u_{r-})+d(X^u_{r-},\mathcal{Z}_{r-}) < a \le b_1(X^u_{r-}) +b_2(X^u_{\theta r})r + d(X^u_{r-},\mathcal{Z}_{r-})\}}\Big) M(dr,du,da,d\theta).
\end{align*}
Note that the integral appearing on the right hand side of the above inequality is a nondecreasing function of $s$. Therefore, taking the supremum over $[0,t]$ and then taking expectations on both sides of the inequality, we arrive at
\begin{align*}
    \mathbb{E}\left[\sup\limits_{0 \le s \le t}\langle \mathcal{Z}_s,1 \rangle^2\right] &\le \mathbb{E}\left[\langle \mathcal{Z}_0,1 \rangle^2\right] + \int_0^t \mathbb{E}\left[\left(B_1 \langle\mathcal{Z}_{s-},1 \rangle + B_2s \int_0^1 \langle\mathcal{Z}_{\theta s},1 \rangle d\theta \right) \left(2 \langle\mathcal{Z}_{s-},1 \rangle + 1 \right)\right] ds\\
    &\le \mathbb{E}\left[\langle \mathcal{Z}_0,1 \rangle^2\right] + \int_0^t \mathbb{E}\left[(B_1 + B_2s) \left(2\sup\limits_{0 \le r \le s}\langle \mathcal{Z}_r,1 \rangle^2  + \sup\limits_{0 \le r \le s}\langle \mathcal{Z}_r,1 \rangle \right)\right] ds\\
    &\le \mathbb{E}\left[\langle \mathcal{Z}_0,1 \rangle^2\right] + \int_0^t 3 (B_1 + B_2s) \mathbb{E}\left[\sup\limits_{0 \le r \le s}\langle \mathcal{Z}_r,1 \rangle^2\right] ds,
\end{align*}
where, on the last line, we have used the fact that the variables $\langle \mathcal{Z}_r,1 \rangle$ are integer-valued. Using the same reasoning as in the proof of \eqref{ineq_Z}, we then obtain that for every $t \ge 0$
\begin{equation*}
    \mathbb{E}\left[\sup\limits_{0 \le s \le t}\langle \mathcal{Z}_s,1 \rangle^2\right] \le \mathbb{E}\left[\langle \mathcal{Z}_0,1 \rangle^2\right] e^{3\left(B_1t + \frac{B_2}{2}t^2 \right)}.
\end{equation*}
\end{proof}

\section{Large population limit}
\label{Section 4}

In this section, we prove Theorem~\ref{th:convergence} by first showing that the sequence of processes $(\mathcal{Z}^N)_{N \ge 1}$ is tight in Section~\ref{4.1}. Then, in Section~\ref{4.2} we show that any converging subsequence of $(\mathcal{Z}^N)_{N \ge 1}$ converges to a process satisfying~\eqref{eq_Zlim}. Finally, we complete the proof by showing the uniqueness of the solution to~\eqref{eq_Zlim} in Section~\ref{subs:uniqueness}. We also show, in Section~\ref{4.4}, that at any time $t>0$ the limiting measure $\mathcal{Z}_t^\infty$ is absolutely continuous with respect to Lebesgue measure on $\mathbb{R}^2$ and we give the partial integro-differential equation satisfied by the density.

\subsection{Tightness}
\label{4.1}

We use the classical approach which consists in showing that $(\mathcal{Z}^N)_{N \in \mathbb{N}}$ is tight in $D_{\mathcal{M}(\hat{\mathbb{R}}^2)}[0,\infty)$, where $\hat{\mathbb{R}}^2$ is the one point compactification of $\mathbb{R}^2$ and $\mathcal{M}(\hat{\mathbb{R}}^2)$ is the set of all finite measures on $\hat{\mathbb{R}}^2$, and then showing that no mass in $\mathcal{Z}^N$ goes to infinity as $N$ tends to infinity, so that $(\mathcal{Z}^N)_{N \in \mathbb{N}}$ is tight in $D_{\mathcal{M}(\mathbb{R}^2)}[0,\infty)$. The arguments used in Section~\ref{4.1.1} to prove the tightness on the compactified space are, once again, standard. However in Section~\ref{4.1.2}, due to the lateral branching, we will need a new argument to control the spatial extent of the population over the time interval $[0,t]$, uniformly in $N$, so that we can show that no mass escapes to infinity.

\subsubsection{Tightness on the compactified space}
\label{4.1.1}

Using Theorem~3.9.1 in \cite{EK}, we show that $(\mathcal{Z}^N)_{N \in \mathbb{N}}$ is tight in $D_{\mathcal{M}(\hat{\mathbb{R}}^2)}[0,\infty)$, by showing that the total mass process can be controlled uniformly in $N$ over any finite time interval and that for every continuous function $f$ on $\hat{\mathbb{R}}^2$, $(\langle \mathcal{Z}_.^N,f\rangle)_{N \in \mathbb{N}}$ is tight. Given that $\mathcal{C}^2(\hat{\mathbb{R}}^2)$ is dense in the set $\mathcal{C}(\hat{\mathbb{R}}^2)$ of all continuous function on $\hat{\mathbb{R}}^2$ for the uniform convergence, it is enough to show that $(\langle \mathcal{Z}_.^N,f\rangle)_{N \in \mathbb{N}}$ is tight for every $f \in \mathcal{C}^2(\hat{\mathbb{R}}^2)$.

Let us prove that the total mass process can be controlled uniformly in $N$ over any finite time interval. Let thus $T \ge 0$, $\varepsilon > 0$ and $N \ge 1$. Using \eqref{ineq_Z2} for the unscaled process $\mathcal{Z}^{(N)}$ defined in the paragraph just before \eqref{normalised L d} and dividing by $N^2$, we have
\begin{equation*}
    \mathbb{E}\left[\sup\limits_{0 \le t \le T}\langle \mathcal{Z}^N_t,1 \rangle^2\right] \le \mathbb{E}\left[\langle \mathcal{Z}^N_0,1 \rangle^2\right] e^{3(B_1T + \frac{B_2}{2}T^2)}.
\end{equation*}
Using Markov's inequality and recalling that $\sup_{N \in \mathbb{N}} \mathbb{E}[\langle \mathcal{Z}_0^N,1 \rangle^2] < + \infty$ by assumption, we obtain that there exists a constant $C_T$ independent of $N$ such that for any $K > 0$,
\begin{align*}
    \mathbb{P}\left(\sup_{t \le T} \ \langle \mathcal{Z}_t^N , 1 \rangle > K\right) \le \frac{C_T}{K^2}.
\end{align*}
From this, we deduce that for every $\varepsilon>0$, there exists $K_\varepsilon>0$ such that for every $N \ge 1$,
\begin{equation*}
    \mathbb{P}\left(\sup_{t \le T} \ \langle \mathcal{Z}_t^N , 1 \rangle > K_\varepsilon\right) \le \varepsilon.
\end{equation*}
Since $\{\nu \in \mathcal{M}(\hat{\mathbb{R}}^2) : \langle \nu,1 \rangle \le K\}$ is a compact subset of $\mathcal{M}(\hat{\mathbb{R}}^2)$, the compact containment condition is satisfied.

Let $f \in\mathcal{C}^2(\hat{\mathbb{R}}^2)$, and let us use the Aldous-Rebolledo criterion \cite{Aldous,Rebolledo} to show that $(\langle \mathcal{Z}_.^N,f\rangle)_{N \in \mathbb{N}}$ is tight. We start by showing that for any $t \ge 0$, $(\langle \mathcal{Z}^N_t ,f \rangle)_{N \ge 1}$ is tight.

Let $t \ge 0$. Using \eqref{ineq_Z} for the unscaled process $\mathcal{Z}^{(N)}$ and dividing by $N$, we have for every $N \ge 1$,
\begin{equation*}
    \mathbb{E}\left[|\langle \mathcal{Z}_t^N,f \rangle|\right] \le \|f\|_{2,\infty}\mathbb{E}\left[\langle \mathcal{Z}_0^N,1 \rangle\right]e^{B_1 t + \frac{B_2}{2} t^2}.
\end{equation*}
Let $\varepsilon > 0$. Using Markov's inequality and the fact that, by the Cauchy-Schwarz inequality, we have
$$
\mathbb{E}\big[\langle \mathcal{Z}_0^N,1 \rangle\big] \le \mathbb{E}\Big[\langle \mathcal{Z}_0^N,1 \rangle^2\Big]^{\frac{1}{2}},
$$
together with the assumption that the second moment of $\langle \mathcal{Z}_0^N,1 \rangle$ is bounded uniformly in $N$, we obtain that there exists a constant $C_t>0$ such that we have for every $N \ge 1$ and every $K>0$,
\begin{equation*}
    \mathbb{P}\left(|\langle \mathcal{Z}_t^N,f \rangle| > K\right) \le \frac{C_t}{K}.
\end{equation*}
We can thus find $K_\varepsilon>0$ large enough such that, for every $N \ge 1$,
\begin{equation*}
    \mathbb{P}\left(|\langle \mathcal{Z}_t^N,f \rangle| > K_\varepsilon \right) \le \varepsilon.
\end{equation*}
Since this is true for any $\varepsilon>0$, we can conclude that $(\langle \mathcal{Z}_t^N,f \rangle)_{N \in \mathbb{N}}$ is indeed tight.

Let us now show that for every $T \ge 0$, every sequence $(\tau_N)_{N \in \mathbb{N}}$ of stopping times bounded by $T$ and every $\varepsilon > 0$, there exists $\delta_0 > 0$ such that
\begin{equation*}
    \sup \limits_{\delta \in [0, \delta_0]}\mathbb{P}\left( \left| \langle \mathcal{Z}^N_{\tau_N + \delta} - \mathcal{Z}^N_{\tau_N}, f \rangle\right| > \varepsilon \right) \le \varepsilon.
\end{equation*}
Let $(\tau_N)_{N \in \mathbb{N}}$ be a sequence of stopping times bounded by $T$, and let us show that for every $N \ge 1$ and $\delta \le 1$,
\begin{equation}\label{control QV}
    \mathbb{E}\left[\left| \langle \mathcal{Z}^N_{\tau_N + \delta} - \mathcal{Z}^N_{\tau_N}, f \rangle\right|\right] \le C_T \sqrt{\delta},
\end{equation}
for some constant $C_T$ independent of $N$ and $\delta$. Indeed, once this inequality is proved, Markov's inequality gives us that for every $\varepsilon > 0$
\begin{equation*}
    \mathbb{P}\left( \left| \langle \mathcal{Z}^N_{\tau_N + \delta} - \mathcal{Z}^N_{\tau_N}, f \rangle\right| > \varepsilon \right) \le \frac{C_T \sqrt{\delta}}{\varepsilon},
\end{equation*}
and setting $\delta_0 = \varepsilon^4/C_T^2$ then yields the result.

Hence, let us show \eqref{control QV}. Using \eqref{eq_Z2} for the unscaled process, dividing by $N$ and recalling our scaling assumptions~\eqref{normalised L d}, we can write for every $N \ge 1$ and $t \ge 0$
\begin{align*}
     &\langle \mathcal{Z}^{N}_t,f \rangle - \langle \mathcal{Z}^{N}_0,f \rangle - \int_0^t\int_{\mathbb{R}^2} \Bigg[\sum\limits_{i=1}^2 \left(\frac{\partial f}{\partial x_i}(x) \int_0^s\int_{\mathbb{R}^2} L^{i}_{s-r}(x-y) \mathcal{Z}^{N}_r(dy)dr + \frac{\sigma^2}{2} \frac{\partial^2 f}{\partial x_i^2}(x)\right) \\
    & \hspace{5.5cm}+ (b_1(x) + b_2(x)(t-s)-d(x,\mathcal{Z}^{N}_s))f(x) \Bigg] \mathcal{Z}^{N}_s(dx)ds\\
    &= \frac{1}{N} M_{0,t}^{N,f} \\
    &\quad + \frac{1}{N} \int_{[0,t] \times \mathbb{N} \times \mathbb{R}_+ \times (0,1)} \bigg\{\left(\mathds{1}_{\{u\in V^N_{s-},a \le b_1(X^{N,u}_{s-})\}}-\mathds{1}_{\{u\in V^N_{s-},b_1(X^{N,u}_{s-}) < a \le b_1(X^{N,u}_{s-})+d(X^{N,u}_{s-},\mathcal{Z}^{N}_{s-})\}}\right) f(X^{N,u}_{s-})\\
    & \hspace{4.4cm}+ \mathds{1}_{\{u\in V^N_{\theta s},b_1(X^{N,u}_{s-})+d(X^{N,u}_{s-},\mathcal{Z}^{N}_{s-}) < a \le b_1(X^{N,u}_{s-})+b_2(X^{N,u}_{\theta s})s+d(X^{N,u}_{s-},\mathcal{Z}^{N}_{s-})\}} f(X^{N,u}_{\theta s}) \bigg\}\\
    & \hspace{10cm}\times(M(ds,du,da,d\theta)-ds\,n(du)\,da\,d\theta).
\end{align*}
Consequently, we have
\begin{align}
\label{ineq_tightness}
   \left| \langle \mathcal{Z}^N_{\tau_N + \delta} - \mathcal{Z}^N_{\tau_N}, f \rangle\right| \le & \Bigg| \int_{\tau_N}^{\tau_N + \delta} \int_{\mathbb{R}^2}\Bigg[ \sum_{i=1}^2 \left(\frac{\partial f}{\partial x_i}(x) \int_0^s \int_{\mathbb{R}^2} L^i_{s-r}(x-y)\mathcal{Z}_r^N(dy)dr + \frac{\sigma^2}{2}\frac{\partial^2 f}{\partial x_i^2}(x)\right) \nonumber \\
   & \hspace{2cm}+(b_1(x) + b_2(x)(\tau_N+\delta -s)-d(x,\mathcal{Z}^N_s))f(x) \Bigg] \mathcal{Z}_s^N(dx)ds\Bigg| \nonumber \\
   & + \frac{1}{N} \left|M_{\tau_N,\tau_N + \delta}^{N,f}\right|  \\
   & + \frac{1}{N} \Bigg|\int_{(\tau_N,\tau_N + \delta] \times \mathbb{N} \times \mathbb{R}_+ \times (0,1)} \Big(\mathds{1}_{\{u\in V^N_{s-},a \le b_1(X^{N,u}_{s-})\}}f(X^{N,u}_{s-}) \nonumber  \\
   & \hspace{1cm}-\mathds{1}_{\{u\in V^N_{s-},b_1(X^{N,u}_{s-}) < a \le b_1(X^{N,u}_{s-})+d(X^{N,u}_{s-},\mathcal{Z}^N_{s-})\}}f(X^{N,u}_{s-})\nonumber \\
   & \hspace{1cm} + \mathds{1}_{\{u\in V^N_{\theta s},b_1(X^{N,u}_{s-})+d(X^{N,u}_{s-},\mathcal{Z}^N_{s-}) < a \le b_1(X^{N,u}_{s-})+b_2(X^{N,u}_{\theta s})s+d(X^{N,u}_{s-},\mathcal{Z}^N_{s-})\}} f(X^{N,u}_{\theta s}) \Big)\nonumber \\
   &\hspace{6.5cm}\times (M(ds,du,da,d\theta)-ds\,n(du)\,da\,d\theta)\Bigg| \nonumber.
\end{align}
First, using Assumption \eqref{L_1}, we can write for every $s \in [\tau_N, \tau_N + \delta]$ and every $x \in \mathbb{R}^2$
\begin{align*}
    & \left| \sum_{i=1}^2 \left(\frac{\partial f}{\partial x_i}(x) \int_0^s \int_{\mathbb{R}^2} L^i_{s-r}(x-y)\mathcal{Z}_r^N(dy)dr + \frac{\sigma^2}{2}\frac{\partial^2 f}{\partial x_i^2}(x)\right) +(b_1(x) + b_2(x)(\tau_N+\delta-s)-d(x,\mathcal{Z}^N_s))f(x) \right|\\
     &\le \|f\|_{2,\infty} \left(\int_0^s h_1(s-r)\langle \mathcal{Z}_r^N , 1 \rangle dr+ \sigma^2 + B_1+B_2(T+\delta) + C\langle \mathcal{Z}_s^N , 1 \rangle \right) \\
     &\le \|f\|_{2,\infty} \left(\sigma^2 + B_1 +B_2(T+\delta)+ \left(C+\int_0^s h_1(s-r) dr\right)\sup\limits_{t \le T+\delta}\langle \mathcal{Z}_t^N , 1 \rangle \right).
\end{align*}
Thus, integrating against $\mathcal{Z}^N_s$ and using the fact that $h_1$ is integrable, we can write for every $s \in [\tau_N, \tau_N + \delta]$
\begin{align*}
    &\int_{\mathbb{R}^2}\Bigg| \sum_{i=1}^2 \left(\frac{\partial f}{\partial x_i}(x) \int_0^s \int_{\mathbb{R}^2} L^i_{s-r}(x-y)\mathcal{Z}_r^N(dy)dr + \frac{\sigma^2}{2}\frac{\partial^2 f}{\partial x_i^2}(x)\right) \\
    & \qquad \quad + (b_1(x) + b_2(x)(\tau_N+\delta -s)-d(x,\mathcal{Z}^N_s))f(x) \Bigg| \mathcal{Z}_s^N(dx) \\
    &\le \|f\|_{2,\infty} \left(\sigma^2 + B_1+B_2(T+\delta) + C' \sup\limits_{t \le T+\delta}\langle \mathcal{Z}_t^N , 1 \rangle \right)\sup\limits_{t \le T+\delta}\langle \mathcal{Z}_t^N , 1 \rangle.
\end{align*}
Integrating over $[\tau_N,\tau_N + \delta]$, taking expectation and using \eqref{ineq_Z} and \eqref{ineq_Z2}, we obtain
\begin{align}
\label{ineq_1}
    &\mathbb{E}\Bigg[\Bigg| \int_{\tau_N}^{\tau_N + \delta} \int_{\mathbb{R}^2}\Bigg[ \sum_{i=1}^2 \left(\frac{\partial f}{\partial x_i}(x) \int_0^s \int_{\mathbb{R}^2} L^i_{s-r}(x-y)\mathcal{Z}_r^N(dy)dr + \frac{\sigma^2}{2}\frac{\partial^2 f}{\partial x_i^2}(x)\right) \nonumber \\
    &\hspace{2.5cm}+(b_1(x) + b_2(x)(\tau_N+\delta-s)-d(x,\mathcal{Z}^N_s))f(x) \Bigg]\mathcal{Z}_s^N(dx)ds\Bigg|\Bigg] \nonumber\\
    &\le C_T \delta \|f\|_{2,\infty} \mathbb{E}\left[\left(1 +\sup\limits_{t \le T+\delta}\langle \mathcal{Z}_t^N , 1 \rangle \right)\sup\limits_{t \le T+\delta}\langle \mathcal{Z}_t^N , 1 \rangle\right] \nonumber\\
    &\le C'_T\delta \|f\|_{2,\infty} e^{3(B_1(T+\delta)+\frac{B_2}{2}(T+\delta)^2)}.
\end{align}
This gives us appropriate control on the first term on the right hand side of \eqref{ineq_tightness}.
Considering the unscaled process $\mathcal{Z}^{(N)}$ and using \eqref{ineq_M}, we can then write
\begin{equation*}
    \left\langle M^{N,f}_{\tau_N,\tau_N + .} \right\rangle_{\delta} \le \delta \sigma^2 \|f\|_{2,\infty}^2 (\langle \mathcal{Z}^{(N)}_0,1 \rangle + N_{0,t}).
\end{equation*}
Taking expectations on both sides of this inequality and using \eqref{ineq_N}, we obtain
\begin{align*}
    \mathbb{E}\left[ \left\langle M_{\tau_N,\tau_N + .}^{N,f} \right\rangle_{\delta}\right] & \le \delta \sigma^2 \|f\|_{2,\infty}^2  \mathbb{E}\left[\langle \mathcal{Z}^{(N)}_0,1 \rangle + N_{0,T+\delta}\right]\\
    & \le \delta \sigma^2 \|f\|_{2,\infty}^2 \mathbb{E}\left[\langle \mathcal{Z}^{(N)}_0,1 \rangle\right]\left(1 + \left(B_1(T+\delta) + \frac{B_2}{2}(T+\delta)^2\right) e^{B_1(T+\delta)+\frac{B_2}{2} (T+\delta)^2}\right),
\end{align*}
and, noting that $\dfrac{1}{N^2}\langle\mathcal{Z}^{(N)}_t,1\rangle = \dfrac{1}{N}\langle\mathcal{Z}^{N}_t,1\rangle$, we arrive at
\begin{align}
\label{ineq_2}
    \mathbb{E}\left[ \left\langle \frac{1}{N} M_{\tau_N,\tau_N + .}^{N,f}  \right\rangle_{\delta} \right] &\le \frac{1}{N} \delta \sigma^2 \|f\|_{2,\infty}^2 \mathbb{E}\left[\langle \mathcal{Z}^{N}_0,1 \rangle\right]\left(1 + \left(B_1(T+\delta) + \frac{B_2}{2}(T+\delta)^2\right) e^{B_1(T+\delta)+\frac{B_2}{2} (T+\delta)^2}\right) \nonumber \\
    &\le \frac{1}{N} C_T \delta \left(1+e^{B_1(T+\delta)+\frac{B_2}{2} (T+\delta)^2}\right).
\end{align}
Finally, we have
\begin{align*}
    & \mathbb{E}\Bigg[\Bigg( \frac{1}{N} \int_{[\tau_N,\tau_N + \delta] \times \mathbb{N} \times \mathbb{R}_+ \times (0,1)}\Bigg(\left(\mathds{1}_{\{u\in V^N_{s-},a \le b_1(X^{N,u}_{s-})\}}-\mathds{1}_{\{u\in V^N_{s-},b_1(X^{N,u}_{s-}) < a \le b_1(X^{N,u}_{s-})+d(X^{N,u}_{s-},\mathcal{Z}^N_{s-})\}}\right) f(X^{N,u}_{s-}) \\
    & \hspace{5cm} + \mathds{1}_{\{u\in V^N_{\theta s},b_1(X^{N,u}_{s-})+d(X^{N,u}_{s-},\mathcal{Z}^N_{s-}) < a \le b_1(X^{N,u}_{s-})+b_2(X^{N,u}_{\theta s})s+d(X^{N,u}_{s-},\mathcal{Z}^N_{s-})\}} f(X^{N,u}_{\theta s}) \Bigg) \\
    & \hspace{10cm}\times (M(ds,du,da,d\theta)-ds\, n(du) \,da\, d\theta) \Bigg)^2 \ \Bigg] \\
    &= \mathbb{E}\Bigg[\frac{1}{N^2} \int_{[\tau_N,\tau_N + \delta] \times \mathbb{N} \times \mathbb{R}_+ \times (0,1)}\Bigg(\left(\mathds{1}_{\{u\in V^N_{s-},a \le b_1(X^{N,u}_{s-})\}}-\mathds{1}_{\{u\in V^N_{s-},b_1(X^{N,u}_{s-}) < a \le b_1(X^{N,u}_{s-})+d(X^{N,u}_{s-},\mathcal{Z}^N_{s-})\}}\right) f(X^{N,u}_{s-}) \\
    & \hspace{5.4cm} + \mathds{1}_{\{u\in V^N_{\theta s},b_1(X^{N,u}_{s-})+d(X^{N,u}_{s-},\mathcal{Z}^N_{s-}) < a \le b_1(X^{N,u}_{s-})+b_2(X^{N,u}_{\theta s})s+d(X^{N,u}_{s-},\mathcal{Z}^N_{s-})\}} f(X^{N,u}_{\theta s}) \Bigg)^2\\
    & \hspace{14.5cm}\times ds\, n(du) \,da\, d\theta \ \Bigg].
\end{align*}
Let
\begin{align*}
    H(s,u,a,\theta) = & \left(\mathds{1}_{\{u\in V^N_{s-},a \le b_1(X^{N,u}_{s-})\}}-\mathds{1}_{\{u\in V^N_{s-},b_1(X^{N,u}_{s-}) < a \le b_1(X^{N,u}_{s-})+d(X^{N,u}_{s-},\mathcal{Z}^N_{s-})\}}\right) f(X^{N,u}_{s-}) \nonumber\\
    & + \mathds{1}_{\{u\in V^N_{\theta s},b_1(X^{N,u}_{s-})+d(X^{N,u}_{s-},\mathcal{Z}^N_{s-}) < a \le b_1(X^{N,u}_{s-})+b_2(X^{N,u}_{\theta s})s+d(X^{N,u}_{s-},\mathcal{Z}^N_{s-})\}} f(X^{N,u}_{\theta s}).
\end{align*}
We have
\begin{equation*}
H(s,u,a,\theta) =
\begin{cases}
    f(X^{N,u}_{s-})\quad  \text{ if $u \in V^N_{s-}$ and $a \le b_1(X^{N,u}_{s-})$}, \\
    - f(X^{N,u}_{s-}) \  \text{ if $u \in V^N_{s-}$ and $b_1(X^{N,u}_{s-}) <  a \le b_1(X^{N,u}_{s-})+d(X^{N,u}_{s-},\mathcal{Z}^N_{s-})$},\\
    f(X^{N,u}_{\theta s}) \quad  \text{ if $u \in V^N_{\theta s}$ and} \\
    \hspace{2cm} b_1(X^{N,u}_{s-})+d(X^{N,u}_{s-},\mathcal{Z}^N_{s-}) < a \le b_1(X^{N,u}_{s-})+b_2(X^{N,u}_{\theta s})s+d(X^{N,u}_{s-},\mathcal{Z}^N_{s-}),
\end{cases}
\end{equation*}
so that
\begin{equation*}
H(s,u,a,\theta)^2 =
\begin{cases}
    f(X^{N,u}_{s-})^2 \quad \text{ if $u \in V^N_{s-}$ and $a \le b_1(X^{N,u}_{s-})+d(X^{N,u}_{s-},\mathcal{Z}^N_{s-})$}\\
    f(X^{N,u}_{\theta s})^2 \quad  \text{ if $u \in V^N_{\theta s}$ and} \\
    \hspace{2cm} b_1(X^{N,u}_{s-})+d(X^{N,u}_{s-},\mathcal{Z}^N_{s-}) < a \le b_1(X^{N,u}_{s-})+b_2(X^{N,u}_{\theta s})s+d(X^{N,u}_{s-},\mathcal{Z}^N_{s-}).
\end{cases}
\end{equation*}
We thus have
\begin{align}
\label{ineq_3}
    & \mathbb{E}\Bigg[\Bigg( \frac{1}{N} \int_{[\tau_N,\tau_N + \delta] \times \mathbb{N} \times \mathbb{R}_+ \times (0,1)}\Bigg(\left(\mathds{1}_{\{u\in V^N_{s-},a \le b_1(X^{N,u}_{s-})\}}-\mathds{1}_{\{u\in V^N_{s-},b_1(X^{N,u}_{s-}) < a \le b_1(X^{N,u}_{s-})+d(X^{N,u}_{s-},\mathcal{Z}^N_{s-})\}}\right) f(X^{N,u}_{s-}) \nonumber\\
    & \hspace{5cm} + \mathds{1}_{\{u\in V^N_{\theta s},b_1(X^{N,u}_{s-})+d(X^{N,u}_{s-},\mathcal{Z}^N_{s-}) < a \le b_1(X^{N,u}_{s-})+b_2(X^{N,u}_{\theta s})s+d(X^{N,u}_{s-},\mathcal{Z}^N_{s-})\}} f(X^{N,u}_{\theta s}) \Bigg)\nonumber \\
    & \hspace{10cm}\times (M(ds,du,da,d\theta)-ds\, n(du) \,da\, d\theta) \Bigg)^2 \ \Bigg] \nonumber\\
    & = \frac{1}{N^2}\mathbb{E}\Bigg[\int_{\tau_N}^{\tau_N + \delta}\int_{\mathbb{N} \times \mathbb{R}_+\times (0,1)} \Big(\mathds{1}_{\{u \in V^N_{s-},a\le b_1(X_{s-}^{N,u})+d(X_{s-}^{N,u},\mathcal{Z}^N_{s-})\}} f(X_{s-}^{N,u})^2 \nonumber \\
    & \hspace{1.5cm}+ \mathds{1}_{\{u\in V^N_{\theta s},b_1(X_{s-}^{N,u})+d(X_{s-}^{N,u},\mathcal{Z}^N_{s-}) < a \le b_1(X_{s-}^{N,u})+b_2(X_{\theta s}^{N,u})s+d(X_{s-}^{N,u},\mathcal{Z}^N_{s-})\}} f(X^{N,u}_{\theta s})^2\Big) ds\,n(du)\,da\,d\theta \Bigg] \nonumber\\
    & \le \frac{1}{N^2}\mathbb{E}\left[\int_{\tau_N}^{\tau_N + \delta} \left( (B_1+C\langle \mathcal{Z}^N_s, 1 \rangle) \langle \mathcal{Z}^{(N)}_s, f^2\rangle + B_2s \int_0^1\langle \mathcal{Z}^{(N)}_{\theta s}, f^2\rangle d\theta\right) ds\right] \nonumber\\
    &\le \frac{1}{N}\delta\|f\|_{2,\infty}^2 \mathbb{E}\left[ B_1+B_2(T+\delta) \sup\limits_{s \le T+\delta}\langle \mathcal{Z}^N_s, 1\rangle + C \sup\limits_{s \le T+\delta}\langle \mathcal{Z}^N_s, 1\rangle ^2\right] \nonumber\\
    &\le \frac{1}{N}C_T\delta (1+e^{B_1(T+\delta)+ \frac{B_2}{2}(T+\delta)^2}).
\end{align}
Taking expectations on both sides of \eqref{ineq_tightness}, using \eqref{ineq_1} and applying Cauchy-Schwarz' inequality to \eqref{ineq_2} and \eqref{ineq_3}, we obtain for every $\delta \in (0,1)$
\begin{equation*}
    \mathbb{E}\left[ \left| \langle \mathcal{Z}^N_{\tau_N + \delta} - \mathcal{Z}^N_{\tau_N}, f \rangle\right|\right] \le C_T \sqrt{\delta},
\end{equation*}
and we can conclude that $(\langle \mathcal{Z}_.^N,f\rangle)_{N \in \mathbb{N}}$ is tight for all $f \in \mathcal{C}^2(\hat{\mathbb{R}}^2)$ and consequently that $(\mathcal{Z}^N)_{N \in \mathbb{N}}$ is tight in $D_{\mathcal{M}(\hat{\mathbb{R}}^2)}[0,\infty)$.

\subsubsection{Uniform control on the spatial extent of the population}
\label{4.1.2}

To prove that $(\mathcal{Z}^N)_{N \in \mathbb{N}}$ is tight in $D_{\mathcal{M}(\mathbb{R}^2)}[0,\infty)$ it suffices to prove that at any given time $t\geq 0$, no mass escapes to infinity as $N$ tends to infinity (see Corollary 3.9.3 in \cite{EK}). More precisely, we prove the following result.
\begin{Prop}\label{prop: mass escaping}
For every $t\geq 0$ and every $\varepsilon>0$, there exists $R>0$ such that
$$
\limsup_{N\rightarrow \infty} \mathbb{P}\left(\langle \mathcal{Z}_t^N,\mathds{1}_{\{|x| > R\}}\rangle > \varepsilon \right) \leq \varepsilon.
$$
\end{Prop}
\begin{proof}[Proof of Proposition~\ref{prop: mass escaping}]
We proceed in two steps. We start by considering the case where $b_1$ and $b_2$ are constant functions and using a fine coupling with a process with identical branching rates, $L\equiv 0$ and $d\equiv 0$ to control the spatial extent of $\mathcal{Z}^N_t$ uniformly in $N$. We then argue that the proof extends to the general case where $b_1$ and $b_2$ are bounded (by some constants $B_1$ and $B_2$) using a comparison argument that involves a simpler coupling with a process in which there are no deaths and the apical and lateral branching rates are respectively equal to the constants $B_1$ and $B_2$. The coupling is only used to control the total number of individuals in the time interval $[0,t]$, which in turn will control the absolute value of the drift experienced by each particle (and therefore its location at time $t$).

Thus, suppose as a start that the branching rates $b_1$ and $b_2$ are positive constants. Let $t \ge 0$, and let us first bound $\mathbb{E}\left[\langle \mathcal{Z}_t^N,\mathds{1}_{\{|x| > R\}}\rangle\right]$ uniformly in $N$, and then show that this bound goes to $0$ as $R$ tends to infinity. Using Markov's inequality, this will allow us to conclude.

Let $R > 0$ and $A > 0$. For every $N \ge 1$, we have
\begin{equation}
\label{ineq_AR}
    \mathbb{E}\left[\langle \mathcal{Z}_t^N,\mathds{1}_{\{|x| > R\}}\rangle\right] = \mathbb{E}\left[\langle \mathcal{Z}_t^N,\mathds{1}_{\{|x| > R\}}\rangle \mathds{1}_{\{\sup\limits_{s \le t}\langle  \mathcal{Z}_s^N,1 \rangle > A\}} + \langle \mathcal{Z}_t^N,\mathds{1}_{\{|x| > R\}}\rangle \mathds{1}_{\{\sup\limits_{s \le t}\langle  \mathcal{Z}_s^N,1 \rangle \le A\}}\right].
\end{equation}
By the Cauchy-Schwarz inequality, we have
\begin{equation}
\label{ineq_AR1}
    \mathbb{E}\left[\langle \mathcal{Z}_t^N,\mathds{1}_{\{|x| > R\}}\rangle \mathds{1}_{\{\sup\limits_{s \le t}\langle  \mathcal{Z}_s^N,1 \rangle > A\}}\right] \le \mathbb{E}\left[\langle \mathcal{Z}_t^N,1\rangle^2\right]^{\frac{1}{2}} \  \mathbb{P}\bigg(\sup\limits_{s \le t} \ \langle  \mathcal{Z}_s^N,1 \rangle > A\bigg)^{\frac{1}{2}}.
\end{equation}
Using the Markov inequality and \eqref{ineq_Z2}, we obtain that
\begin{equation*}
    \mathbb{P}\left(\sup\limits_{s \le t} \ \langle  \mathcal{Z}_s^N,1 \rangle > A\right) \le \frac{1}{A} \mathbb{E}\left[\sup\limits_{s \le t} \ \langle \mathcal{Z}_t^N,1\rangle^2\right]^{\frac{1}{2}} \le \frac{C_t}{A},
\end{equation*}
for a constant $C_t>0$ independent of $N$. Thus, combining \eqref{ineq_AR1} with \eqref{ineq_Z2}, we can write
\begin{equation*}
    \mathbb{E}\left[\langle \mathcal{Z}_t^N,\mathds{1}_{\{|x| > R\}}\rangle \mathds{1}_{\{\sup\limits_{s \le t}\langle  \mathcal{Z}_s^N,1 \rangle > A\}}\right] \le \frac{C'_t}{\sqrt{A}}.
\end{equation*}

This allows us to control the first term on the right hand side of \eqref{ineq_AR}. We now turn to the second term. For every $N\geq 1$, let
\begin{equation}
\big(\widetilde{\mathcal{Z}}_t^{(N)}\big)_{t\geq 0} = \Bigg( \sum_{u\in V^N_t(\widetilde{\mathcal{Z}})} \delta_{\widetilde{X}^{N,u}_t}\Bigg)_{t\geq 0} \label{base process}
\end{equation}
be the (unscaled) process corresponding to $L \equiv 0$ and $d \equiv 0$, using the same realisations of the individual Brownian motions and of the Poisson random measure $M$ as $\mathcal{Z}^{(N)}$ with the exception that death events are cancelled (that is, if $M(\{(s,u,a,\theta)\})=1$, $u\in V_{s-}^N(\widetilde{\mathcal{Z}})$ and $a\in (b_1,b_1+d(X^u_{s-},\mathcal{Z}^{(N)}_{s-}/N)]$, then nothing happens for $\widetilde{\mathcal{Z}}_s^{(N)}$).
Note that we write $d(X^u_{s-},\mathcal{Z}^{(N)}_{s-}/N)$ and not $d(X^u_{s-},\widetilde{\mathcal{Z}}^{(N)}_{s-}/N)$ here, in order to use the same indicator functions to decide for apical and lateral branching events in $\mathcal{Z}^{(N)}$ and in the subset of $\widetilde{\mathcal{Z}}^{(N)}$ that we couple with $\mathcal{Z}^{(N)}$ below.

We start by showing that for every $N\ge 1$, $t \ge 0$ and $u \in V^N_t$, we can write
\begin{equation}
\label{eq_XW}
    X^{N,u}_t = \widetilde{X}^{N,u}_t + \int_{I^{N,u}_t} \int_0^s\int_{\mathbb{R}^2}L_{s-r}(X_s^{N,u}-x)\mathcal{Z}^N_r(dx)drds,
\end{equation}
where $I^{N,u}_t$ is the set of times $s \le t$ at which individual $u$ is in $V^N_s(\mathcal{Z})$ or has an ancestor in $V^N_s(\mathcal{Z})$, and $X^{N,u}_s$ is the position of individual $u$ or the ancestor of $u$ alive at time $s \in I^{N,u}_t$.

We show by induction on $k \ge 1$ that for every $t \in [0,T^N_k(\mathcal{Z}))$, $V^N_t(\mathcal{Z})\subset V^N_t(\widetilde{\mathcal{Z}})$ and \eqref{eq_XW} holds for every $u \in V^N_t(\mathcal{Z})$. Notice that because $\widetilde{\mathcal{Z}}^{(N)}$ corresponds to the process with $d \equiv 0$, individuals that have died in the process $\mathcal{Z}^{(N)}$ can keep branching in the process $\widetilde{\mathcal{Z}}^{(N)}$. Thus, we need to relabel the individuals to keep track of the different lines of descent in order to establish a correspondence between the indices in $V^N_t(\mathcal{Z})$ and $V^N_t(\widetilde{\mathcal{Z}})$.

Let $k = 1$ and $t \in [T^N_0(\mathcal{Z}),T^N_1(\mathcal{Z}))$. The branching rates are constant in space so that an individual $u \in V^N_0$ branches for the process $\mathcal{Z}^{(N)}$ at time $t$ if and only if it branches for the process $\widetilde{\mathcal{Z}}^{(N)}$ at time~$t$. We then have that $V^N_t(\mathcal{Z}) = V^N_t(\widetilde{\mathcal{Z}}) = V^N_0$ for any $t \in [T^N_0(\mathcal{Z}),T^N_1(\mathcal{Z}))$ and $I^{N,u}_t = [0,t]$ for any $u \in V^N_0$. For every $u \in V^N_0$, we have
\begin{equation*}
    dX^{N,u}_t - \left(\int_0^t\int_{\mathbb{R}^2}L_{t-s}(X_t^{N,u}-x)\mathcal{Z}^N_s(dx)ds\right)dt = \sigma dW^u_t = d\widetilde{X}^{N,u}_t,
\end{equation*}
and $X^{N,u}_0 = \widetilde{X}^{N,u}_0$. Thus \eqref{eq_XW} holds for every $t \in [T^N_0(\mathcal{Z}),T^N_1(\mathcal{Z}))$ and every $u \in V^N_t(\mathcal{Z})$.

Let $k \ge 1$, and suppose that for every $t \in [0,T^N_k(\mathcal{Z}))$, $V^N_t(\mathcal{Z})\subset V^N_t(\widetilde{\mathcal{Z}})$ and \eqref{eq_XW} holds for every $u \in V^N_t(\mathcal{Z})$.

We have $V^N_{T_k-}(\mathcal{Z})\subset V^N_{T_k-}(\widetilde{\mathcal{Z}})$. If individual $U_k$ dies at time $T^N_k(\mathcal{Z})$ for the process $\mathcal{Z}^{(N)}$, $U^N_k$ is removed from $V^N_{T_k}(\mathcal{Z})$ but remains in $V^N_{T_k}(\widetilde{\mathcal{Z}})$. If individual $U^N_k$ branches at time $T^N_k(\mathcal{Z})$ for the process $\mathcal{Z}^{(N)}$, it also branches at time $T^N_k(\mathcal{Z})$ for the process $\widetilde{\mathcal{Z}}^{(N)}$. Thus, we have $V^N_{T_k}(\mathcal{Z})\subset V^N_{T_k}(\widetilde{\mathcal{Z}})$. Furthermore, no individual in $V^N_{T_k}(\mathcal{Z})$ branches or dies during $(T^N_k(\mathcal{Z}),T^N_{k+1}(\mathcal{Z}))$ for either processes $\mathcal{Z}^{(N)}$ or $\widetilde{\mathcal{Z}}^{(N)}$, so that for every $t \in [T^N_k(\mathcal{Z}),T^N_{k+1}(\mathcal{Z}))$, $V^N_t(\mathcal{Z})\subset V^N_t(\widetilde{\mathcal{Z}})$. For every $t \in [T^N_k(\mathcal{Z}),T^N_{k+1}(\mathcal{Z}))$ and $u \in V^N_t(\mathcal{Z})$, we have
\begin{equation*}
    dX^{N,u}_t - \left(\int_0^t\int_{\mathbb{R}^2}L_{t-s}(X_t^{N,u}-x)\mathcal{Z}^N_s(dx)ds\right)dt = \sigma dW^u_t = d\widetilde{X}^{N,u}_t.
\end{equation*}
Let $t \in [T^N_k(\mathcal{Z}),T^N_{k+1}(\mathcal{Z}))$ and $u \in V^N_t(\mathcal{Z})$. There are three possible cases:
\begin{itemize}
    \item[--] $u$ is in $V^N_{T_{k-1}}(\mathcal{Z})$ and we have $X^{N,u}_{T^N_k(\mathcal{Z})} = X^{N,u}_{T^N_k(\mathcal{Z})-}$, $\widetilde{X}^{N,u}_{T^N_k(\mathcal{Z})} = \widetilde{X}^{N,u}_{T^N_k(\mathcal{Z})-}$ and $I^{N,u}_t = I^{N,u}_{T^N_k(\mathcal{Z})-} \cup [T^N_k(\mathcal{Z}),t]$.
    \item[--] $u$ is born at time $T^N_k(\mathcal{Z})$ from apical branching (for both processes $\mathcal{Z}^N$ and $\widetilde{\mathcal{Z}}^N$) and we have $X^{N,u}_{T^N_k(\mathcal{Z})} = X^{N,U^N_k}_{T^N_k(\mathcal{Z})-}$ and $\widetilde{X}^{N,u}_{T^N_k(\mathcal{Z})} = \widetilde{X}^{N,U^N_k}_{T^N_k(\mathcal{Z})-}$, and $I^{N,u}_t = I^{N,U^N_k}_{T^N_k(\mathcal{Z})-} \cup [T^N_k(\mathcal{Z}),t]$.
    \item[--] $u$ is born at time $T^N_k(\mathcal{Z})$ from lateral branching and we have $X^{N,u}_{T^N_k(\mathcal{Z})} = X^{N,U^N_k}_{\theta^N_k T^N_k(\mathcal{Z})}$, $\widetilde{X}^{N,u}_{T^N_k(\mathcal{Z})} = \widetilde{X}^{N,U^N_k}_{\theta^N_k T^N_k(\mathcal{Z})}$, and $I^{N,u}_t = I^{N,U^N_k}_{\theta^N_k T^N_k(\mathcal{Z})} \cup [T^N_k(\mathcal{Z}),t]$.
\end{itemize}
In each case, using the induction hypothesis, we can write
\begin{equation*}
    X^{N,u}_{T^N_k(\mathcal{Z})} - \int_{I^{N,u}_{T^N_k}} \int_0^s\int_{\mathbb{R}^2}L_{s-r}(X_s^{N,u}-x)\mathcal{Z}^N_r(dx)drds = \widetilde{X}^{N,u}_{T_k(\mathcal{Z})}.
\end{equation*}
Thus for every $t \in [T^N_k(\mathcal{Z}),T^N_{k+1}(\mathcal{Z}))$ and every $u \in V^N_t(\mathcal{Z})$, Equation \eqref{eq_XW} holds.

We now use \eqref{eq_XW} to control the term
$$
\mathbb{E}\Big[\big\langle \mathcal{Z}_t^N,\mathds{1}_{\{|x| > R\}}\big\rangle\, \mathds{1}_{\{\sup_{s \le t}\langle  \mathcal{Z}_s^N,1 \rangle \le A\}}\Big].$$
On the event $\{\sup_{s \le t} \ \langle  \mathcal{Z}_s^N,1 \rangle \le A\}$, we have
\begin{equation*}
    \left| \int_{I^u_t} \int_0^s\int_{\mathbb{R}^2}L_{s-r}(X_s^{N,u}-x)\mathcal{Z}^N_r(dx)drds\right| \le A \int_0^t \int_0^s h_1(s-r) dsdr \le C_t A.
\end{equation*}
Thus, for every $u \in V^N_t(\mathcal{Z})$,
\begin{equation*}
    |X^{N,u}_t| \le |\widetilde{X}^{N,u}_t| + C_tA,
\end{equation*}
and, using the fact that $V^N_t(\mathcal{Z}) \subset V^N_t(\widetilde{\mathcal{Z}})$, we obtain
\begin{equation*}
    \mathbb{E}\left[\langle \mathcal{Z}_t^N,\mathds{1}_{\{|x| > R\}}\rangle \mathds{1}_{\{\sup\limits_{s \le t}\langle  \mathcal{Z}_s^N,1 \rangle \le A\}}\right] \le \mathbb{E}\left[ \langle \widetilde{\mathcal{Z}}_t^N, \mathds{1}_{\{|x| > R-C_tA\}} \rangle  \right],
\end{equation*}
where $\widetilde{\mathcal{Z}}^N$ is the scaled process $(1/N)\widetilde{\mathcal{Z}}^{(N)}$. Now, we have
\begin{align*}
    \mathbb{E}\left[ \langle \widetilde{\mathcal{Z}}_t^N, \mathds{1}_{\{|x| > R-C_tA\}} \rangle  \right] & = \frac{1}{N} \, \mathbb{E}\left[ \sum\limits_{u \in V^N_t(\widetilde{\mathcal{Z}})} \mathds{1}_{\{|\widetilde{X}^{N,u}_t| > R-C_tA\}}  \right]\\
    & = \frac{1}{N} \, \mathbb{E}\left[\mathbb{E}\left[ \sum\limits_{u \in V^N_t(\widetilde{\mathcal{Z}})} \mathds{1}_{\{|\widetilde{X}^{N,u}_t| > R-C_tA\}} \Bigg| M, \mathcal{Z}^N_0 \right]\right] \\
    & = \frac{1}{N} \, \mathbb{E}\left[\sum\limits_{u \in V^N_t(\widetilde{\mathcal{Z}})} \mathbb{E}\left[  \mathds{1}_{\{|\widetilde{X}^{N,u}_t| > R-C_tA\}} \Big| M, \mathcal{Z}^N_0 \right]\right].
\end{align*}
For every $s\geq 0$ and $x\in \mathbb{R}^2$, let
\begin{equation*}
    \psi_{R,A}(s,x) = \mathbb{P}\left(|B_s + x| > R-C_tA\right),
\end{equation*}
and
\begin{equation*}
    \widetilde{\psi}_{R,A}(s,x) = \mathbb{P}\left(\sup\limits_{0 \le r \le s}|B_r + x| > R-C_tA\right),
\end{equation*}
where $(B_s)_{s \ge 0}$ denotes Brownian motion with variance matrix $\sigma^2\mathrm{Id}$. We can write
\begin{equation*}
    \mathbb{E}\left[  \mathds{1}_{\{|\widetilde{X}^{N,u}_t| > R-C_tA\}} \left| M, \mathcal{Z}^N_0 \right.\right] = \psi_{R,A}(T^{N,u}_t,X^{N,u}_0) \le \widetilde{\psi}_{R,A}(t,X^{N,u}_0),
\end{equation*}
where $T^{N,u}_t$ is the length of $I^{N,u}_t$ and $X^{N,u}_0$ is the ancestor of $u$ alive at time $0$. We then have
\begin{equation*}
     \mathbb{E}\left[ \langle \widetilde{\mathcal{Z}}_t^N, \mathds{1}_{\{|x| > R-C_tA\}} \rangle  \right] \le \frac{1}{N} \, \mathbb{E}\left[\sum\limits_{u \in V^N_t(\widetilde{\mathcal{Z}})} \widetilde{\psi}_{R,A}(t,X^{N,u}_0) \right].
\end{equation*}
Let $V^{N,u}_t \subset V^N_t(\widetilde{\mathcal{Z}})$ be the index set of all individuals alive at time $t$ for the process $\widetilde{\mathcal{Z}}^N$ who are descended from $u \in V^N_0$. We can write
\begin{align*}
    \frac{1}{N} \, \mathbb{E}\left[\sum\limits_{u \in V^N_t(\widetilde{\mathcal{Z}})} \widetilde{\psi}_{R,A}(t,X^{N,u}_0) \right] &= \frac{1}{N} \, \mathbb{E}\left[\sum\limits_{u \in V^N_0} \sum\limits_{v \in V^{N,u}_t} \widetilde{\psi}_{R,A}(t,X^{N,u}_0) \right] \\
    &= \frac{1}{N} \,  \mathbb{E}\left[\sum\limits_{u \in V^N_0} \left(\# {V}^{N,u}_t \right) \ \widetilde{\psi}_{R,A}(t,X^{N,u}_0) \right],
\end{align*}
where $\left(\# {V}^{N,u}_t \right)$ is the number of individuals in $V^{N,u}_t$. We then have, using \eqref{ineq_Z},
\begin{align}
    \mathbb{E}\left[ \langle \widetilde{\mathcal{Z}}_t^N, \mathds{1}_{\{|x| > R-C_tA\}} \rangle  \right] & \le \frac{1}{N} \, \mathbb{E}\left[ \mathbb{E}\left[\sum\limits_{u \in V^N_0} \left(\# {V}^{N,u}_t\right) \ \widetilde{\psi}_{R,A}(t,X^{N,u}_0) \Bigg| \mathcal{Z}^N_0 \right]\right] \nonumber \\
    & \le  \frac{1}{N} \, \mathbb{E}\left[ \sum\limits_{u \in V^N_0}\mathbb{E}\left[ \# {V}^{N,u}_t  \Big| \mathcal{Z}^N_0  \right] \widetilde{\psi}_{R,A}(t,X^{N,u}_0) \right] \nonumber\\
    & \le  e^{B_1t + \frac{B_2}{2}t^2} \mathbb{E}\left[\langle \mathcal{Z}^N_0, \widetilde{\psi}_{R,A}(t,.) \rangle\right]. \label{control brownian process}
\end{align}

Coming back to \eqref{ineq_AR} and using the different estimates we have derived, we obtain that for every $N \ge 1$,
\begin{equation}\label{control radius}
    \mathbb{E}\left[\langle \mathcal{Z}_t^N,\mathds{1}_{\{|x| > R\}}\rangle\right] \le \frac{C_t}{\sqrt{A}} + e^{B_1t + \frac{B_2}{2}t^2} \mathbb{E}\left[\langle \mathcal{Z}^N_0, \widetilde{\psi}_{R,A}(t,.)\rangle\right].
\end{equation}

Since by assumption, $(\mathcal{Z}_0^N)_{N\geq 1}$ converges in distribution to the deterministic measure $\mathcal{Z}_0\in \mathcal{M}(\mathbb{R}^2)$ and since each $\widetilde{\psi}_{R,A}(t,.)$ is a bounded continuous function on $\mathbb{R}^2$, we can write that
$$
\limsup_{N\rightarrow \infty}\mathbb{E}\left[\langle \mathcal{Z}_0^N, \widetilde{\psi}_{R,A}(t,.) \rangle\right] =  \langle \mathcal{Z}_0, \widetilde{\psi}_{R,A}(t,.)\rangle
$$
which, combined with \eqref{control radius}, gives us
$$
\limsup_{N\rightarrow \infty} \mathbb{E}\left[\langle \mathcal{Z}_t^N,\mathds{1}_{\{|x| > R \}}\rangle\right] \le  \frac{C_t}{\sqrt{A}} + e^{B_1t + \frac{B_2}{2}t^2}\langle \mathcal{Z}_0, \widetilde{\psi}_{R,A}(t,.)\rangle.
$$
Furthermore, for any given $A > 0$, we have for every $x \in \mathbb{R}^2$, $\lim_{R \to +\infty}\widetilde{\psi}_{R,A}(t,x) = 0$. Additionally, for every $R > 0$, $\|\widetilde{\psi}_{R,A}(t,.)\|_{\infty} \le 1$. Since $\mathcal{Z}_0$ puts all its mass in $\mathbb{R}^2$, by dominated convergence we obtain that
$$
\lim_{R \to \infty} \langle \mathcal{Z}_0, \widetilde{\psi}_{R,A}(t,.) \rangle = 0.
$$

We can now conclude in the case of constant $b_1,b_2$. Indeed, let $\eta > 0$, and let $A > 0$ be such that $C_t/\sqrt{A} < \eta/2$. Then there exists $R_0$ large enough such that for every $R > R_0$,
\begin{equation}\label{constant case}
\limsup_{N\rightarrow \infty} \mathbb{E}\left[\langle \mathcal{Z}_t^N,\mathds{1}_{\{|x| > R \}}\rangle\right] \leq  \eta,
\end{equation}
and the result of Proposition~\ref{prop: mass escaping} follows from the Markov inequality.

Now that we have proved the result for $b_1$ and $b_2$ constant, let us come back to the general case. Recall that $b_1$ (\emph{resp.}, $b_2$) is supposed to be bounded by a constant $B_1$ (\emph{resp.}, $B_2$). Because the branching rates now depend on space and individual movements \emph{a priori} depend on the whole history of particle locations, we cannot write down a fine coupling of the process $\mathcal{Z}^N$ with the process $\widetilde{\mathcal{Z}}^N$ corresponding to the case with $d\equiv 0$, $L\equiv 0$ and with apical branching at rate $B_1$ and lateral branching at rate proportional to $B_2$ in such a way that the creation of new apexes happens in the same relative order in $\mathcal{Z}^N$ as in $\widetilde{\mathcal{Z}}^N$ and the trajectory of $\widetilde{X}^{N,u}$ can be used as the noise part of the trajectory of $X^{N,u}$ (for the same $u$). Indeed, the recursive construction provided above fails due to the fact that the total branching rate of an individual alive at time $t$ is not the same in both processes, and so branching events in $\mathcal{Z}^N$ are only a subset of branching events in $\widetilde{\mathcal{Z}}^N$, whose occurrence are dependent on the position of the corresponding individuals (so that branching and individual movement cannot be simply decoupled as in $\widetilde{\mathcal{Z}}^N$). However, we only need a weaker comparison argument to conclude. Indeed, by a classical thinning argument, we can construct $\mathcal{Z}^N$ and $\widetilde{\mathcal{Z}}^N$ on the same probability space in such a way that, with probability one, the branching events occurring in $(\mathcal{Z}^N_s)_{0\leq s\leq t}$ form a subset of the branching events occurring in $(\widetilde{\mathcal{Z}}^N_s)_{0\leq s\leq t}$ (our labelling of new particles is not adapted for this coupling, and therefore we cannot easily describe the label set of each $\mathcal{Z}^N_s$ as a well-characterised subset of the label set of $\widetilde{\mathcal{Z}}^N_s$). Since individuals cannot die in $\widetilde{\mathcal{Z}}^N$, this means that the total number of individuals in $\mathcal{Z}^N_t$ is a.s. bounded from above by the total number of individuals in $\widetilde{\mathcal{Z}}^N_t$, although their locations in $\mathbb{R}^2$ are \emph{a priori} different. Now, the drift experienced by each individual in $\mathcal{Z}^N$ alive at some point in $[0,t]$ is bounded in absolute value by $\sup_{[0,t]}h_1$ times the total mass of $\widetilde{\mathcal{Z}}^N$ at time $t$. Therefore, we can now follow the exact same reasoning as in the case of constant $b_1$ and $b_2$, starting at~\eqref{ineq_AR}. First, the inequality
$$
\mathbb{P}\bigg(\sup_{s\in [0,t]} \langle \mathcal{Z}^N_s,1\rangle >A\bigg) \leq \mathbb{P}\bigg(\sup_{s\in [0,t]} \langle \widetilde{\mathcal{Z}}^N_s,1\rangle >A\bigg)
$$
together with \eqref{ineq_AR1} allows us to control the first part of the expectation in the r.h.s. of \eqref{ineq_AR}. Next, the fact that on the event $\{\sup_{s\in [0,t]} \langle \mathcal{Z}^N_s,1\rangle \leq A\}$, each of the at most $A$ particles in encoded in $\mathcal{Z}^N_t$ has followed a path with drift bounded by $\sup_{[0,t]}h_1\times A$ allows us to bound the expected number of particles whose position at time $t$ is at distance at least $R$ from the origin in the same way as in \eqref{control brownian process}. Combining the two steps, we obtain a formula analogous to \eqref{control radius}, and we can then conclude in the same way.
\end{proof}

As explained before, we have now proved all the properties needed to show the tightness of $(\mathcal{Z}^N)_{N \ge 1}$ in $D_{\mathcal{M}(\mathbb{R}^2)}[0,\infty)$.

\subsection{Characterisation of the limit}
\label{4.2}
We prove the following result.
\begin{Prop}
Any converging subsequence of $(\mathcal{Z}^N)_{N \in \mathbb{N}}$ converges in law to a continuous process $\mathcal{Z}^\infty$ which is solution to: for every $f \in \mathcal{C}_b^2(\mathbb{R}^2)$ and $t \ge 0$,
\begin{align*}
    \langle \mathcal{Z}_t^\infty,f \rangle - \langle \mathcal{Z}_0^\infty,f \rangle - \int_0^t\int_{\mathbb{R}^2} \Bigg[ &\sum\limits_{i=1}^2 \Bigg(\frac{\partial f}{\partial x_i}(x) \int_0^s\int_{\mathbb{R}^2} L^i_{s-r}(x-y) \mathcal{Z}_r^\infty(dy)dr + \frac{\sigma^2}{2} \frac{\partial^2 f}{\partial x_i^2}(x)\Bigg)\\
    & + (b_1(x)+b_2(x)(t-s)-d(x,\mathcal{Z}_s^\infty))f(x) \Bigg] \mathcal{Z}_s^\infty(dx)ds = 0.
\end{align*}
\label{Prop_characterisation}
\end{Prop}

\begin{proof}[Proof of Proposition~\ref{Prop_characterisation}]
For $t \ge 0$, $f\in \mathcal{C}_b^2(\mathbb{R}^2)$ and $\nu \in D_{\mathcal{M}(\mathbb{R}^2)}[0,\infty)$, let us define
\begin{align*}
    \phi_t^f(\nu) = \langle \nu_t,f \rangle - \langle \nu_0,f \rangle - \int_0^t\int_{\mathbb{R}^2} \Bigg[&\sum\limits_{i=1}^2 \left(\frac{\partial f}{\partial x_i}(x) \int_0^s\int_{\mathbb{R}^2} L^i_{s-r}(x-y) \nu_r(dy)dr + \frac{\sigma^2}{2} \frac{\partial^2 f}{\partial x_i^2}(x)\right) \\
    &+ (b_1(x)+b_2(x)(t-s)-d(x,\nu_s))f(x) \Bigg] \nu_s(dx)ds.
\end{align*}
We shall show that for every $f\in \mathcal{C}_b^2(\mathbb{R}^2)$ and $t \ge 0$,
\begin{equation}\label{convergence1}
   \lim\limits_{N \to \infty} \mathbb{E}\left[|\phi_t^f(\mathcal{Z}^N)|\right] = \mathbb{E}\left[|\phi_t^f(\mathcal{Z}^\infty)|\right],
\end{equation}
and
\begin{equation}\label{convergence2}
    \lim\limits_{N \to \infty} \mathbb{E}\left[|\phi_t^f(\mathcal{Z}^N)|\right] = 0,
\end{equation}
so that $\mathbb{E}\left[|\phi_t^f(\mathcal{Z}^\infty)|\right] = 0$.

Let thus $f\in \mathcal{C}_b^2(\mathbb{R}^2)$. The function $(\nu \mapsto \phi_t^f(\nu))$ is not continuous on $D_{\mathcal{M}(\mathbb{R}^2)}[0,\infty)$ for the Skorokhod topology. However, Assumption~\eqref{L_2} and the continuity of $b_1$, $b_2$ and $d$ ensure that it is continuous in $\nu$ on $C([0,\infty),\mathcal{M}(\mathbb{R}^2))$, the space of continuous trajectories for the weak topology on $\mathcal{M}(\mathbb{R}^2)$. By construction, almost surely, we have for every $T \ge 0$,
\begin{equation*}
    \sup\limits_{t \le T} \ \sup\limits_{\|f\|_{\infty} \le 1}  \left| \langle \mathcal{Z}_t^N - \mathcal{Z}_{t-}^N ,f \rangle \right| \le \frac{1}{N},
\end{equation*}
so that any limit point $\mathcal{Z}^\infty$ is almost surely continuous and so $\phi_t^f$ is almost surely continuous in $\mathcal{Z}^\infty$. Hence $(\phi_t^f(\mathcal{Z}^N))_{N \ge 1}$ converges in law to $\phi_t^f(\mathcal{Z}^\infty)$. Furthermore, we have, for every $\nu \in D_{\mathcal{M}(\mathbb{R}^2)}[0,\infty)$,
\begin{equation*}
    | \phi_t^f(\nu) | \le C_{t,f} \sup\limits_{s\le t} \left(\langle \nu_s,1 \rangle +\langle \nu_s,1 \rangle ^2\right),
\end{equation*}
and by \eqref{ineq_Z} and \eqref{ineq_2}
\begin{equation*}
    \mathbb{E}\left[\sup\limits_{s\le t} \left(\langle \mathcal{Z}^N_s,1 \rangle +\langle \mathcal{Z}^N_s,1 \rangle ^2\right)\right] \le C_te^{3(B_1t+\frac{B_2}{2}t^2)}.
\end{equation*}
We can thus conclude that for any given $t\geq 0$, $(\phi_t^f(\mathcal{Z}^N))_{N \in \mathbb{N}} $ is uniformly integrable and
\begin{equation*}
    \lim\limits_{N \to \infty} \mathbb{E}\left[|\phi_t^f(\mathcal{Z}^N)|\right] = \mathbb{E}\left[|\phi_t^f(\mathcal{Z}^\infty)|\right],
\end{equation*}
showing~\eqref{convergence1}.

Let us now prove~\eqref{convergence2}. Using \eqref{eq_Z2}, we can write
\begin{align*}
   \phi_t^f(\mathcal{Z}^N) &= \frac{1}{N} M_{0,t}^{N,f} \\
     &+  \frac{1}{N} \int_{[0,t] \times \mathbb{N} \times \mathbb{R}_+\times (0,1)}\hspace{-5pt} \Big((\mathds{1}_{\{u\in V_{s-},a \le b_1(X^{N,u}_{s-})\}}-\mathds{1}_{\{u\in V_{s-},b_1(X^{N,u}_{s-}) < a \le b_1(X^{N,u}_{s-})+d(X^{N,u}_{s-},\mathcal{Z}^N_{s-})\}}) f(X^{N,u}_{s-})\\
    & \hspace{3.4cm} + \mathds{1}_{\{u\in V_{\theta s},b_1(X^{N,u}_{s-})+d(X^{N,u}_{s-},\mathcal{Z}^N_{s-}) < a \le b_1(X^{N,u}_{s-})+b_2(X^{N,u}_{\theta s})s+d(X^{N,u}_{s-},\mathcal{Z}^N_{s-})\}} f(X^{N,u}_{\theta s}) \Big)\\
    & \hspace{9cm}\times (M(ds,du,da,d\theta)-ds\, n(du)\, da\, d\theta).
\end{align*}
Using \eqref{ineq_2} and \eqref{ineq_3} with $\tau_N = 0$ and $\delta = t$, we have
\begin{equation*}
    \mathbb{E}\left[\phi_t^f(\mathcal{Z}^N)^2\right] \le \frac{C_t}{N} \|f\|_{2,\infty}  \xrightarrow[N \to \infty]{} 0,
\end{equation*}
the Cauchy-Schwarz inequality allows us to conclude.
We have thus proved that for every $f \in \mathcal{C}_b^2(\mathbb{R}^2)$ and $t \ge 0$,
\begin{equation*}
    \mathbb{E}\left[|\phi_t^f(\mathcal{Z}^\infty)|\right]=0,
\end{equation*}
and $\mathcal{Z}^\infty$ satisfies
\begin{align*}
    &\langle \mathcal{Z}_t^\infty,f \rangle - \langle \mathcal{Z}_0^\infty,f \rangle - \int_0^t\int_{\mathbb{R}^2} \Bigg[\sum\limits_{i=1}^2 \Bigg(\frac{\partial f}{\partial x_i}(x) \int_0^s\int_{\mathbb{R}^2} L^i_{s-r}(x-y) \mathcal{Z}_r^\infty(dy)dr + \frac{\sigma^2}{2} \frac{\partial^2 f}{\partial x_i^2}(x)\Bigg)\\
    & \hspace{5cm}+ (b_1(x)+b_2(x)(t-s)-d(x,\mathcal{Z}_s^\infty))f(x) \Bigg] \mathcal{Z}_s^\infty(dx)ds = 0.
\end{align*}
\end{proof}

Using the same arguments, we can show that for any $f \in \mathcal{C}_b^{1,2}([0,T] \times  \mathbb{R}^2$) and $t \ge 0$,
\begin{equation*}
\begin{aligned}
    \langle \mathcal{Z}_t^\infty,f_t \rangle - \langle \mathcal{Z}_0^\infty,f_0 \rangle - \int_0^t\int_{\mathbb{R}^2} \Bigg[\frac{\partial f_s}{\partial s} & + \sum\limits_{i=1}^2 \left(\frac{\partial f_s}{\partial x_i}(x) \int_0^s\int_{\mathbb{R}^2} L^i_{s-r}(x-y) \mathcal{Z}_r^\infty(dy)dr + \frac{\sigma^2}{2} \frac{\partial^2 f_s}{\partial x_i^2}(x)\right) \\
    & + (b_1(x)+b_2(x)(t-s)-d(x,\mathcal{Z}_s^\infty))f_s(x) \Bigg] \mathcal{Z}_s^\infty(dx)ds = 0.
\end{aligned}
\end{equation*}
This generalisation will prove useful in the proof of uniqueness of the limit.

\subsection{Uniqueness of the limit}\label{subs:uniqueness}

Let us now prove the uniqueness in $C([0,\infty),\mathcal{M}(\mathbb{R}^2))$ of the solution to: for every $T\ge 0$, $f \in \mathcal{C}_b^{1,2}([0,T],\mathbb{R}^2)$ and $0\le t \le T$,
\begin{equation}
\label{eq_Zlim_s}
\begin{aligned}
    \langle \mathcal{Z}_t^\infty,f_t \rangle - \langle \mathcal{Z}_0^\infty,f_0 \rangle - \int_0^t\int_{\mathbb{R}^2} \Bigg[\frac{\partial f_s}{\partial s} & + \sum\limits_{i=1}^2 \left(\frac{\partial f_s}{\partial x_i}(x) \int_0^s\int_{\mathbb{R}^2} L^i_{s-r}(x-y) \mathcal{Z}_r^\infty(dy)dr + \frac{\sigma^2}{2} \frac{\partial^2 f_s}{\partial x_i^2}(x)\right) \\
    & + (b_1(x)+b_2(x)(t-s)-d(x,\mathcal{Z}_s^\infty))f_s(x) \Bigg] \mathcal{Z}_s^\infty(dx)ds = 0.
\end{aligned}
\end{equation}
Here again, the technique we employ is classical and based on the mild form of the equation, but we give the details of all the required controls (and in particular those of the additional terms due to lateral branching, which have a less standard form).

Let $\mathcal{Z},\widetilde{\mathcal{Z}} \in C([0,\infty),\mathcal{M}(\mathbb{R}^2))$ satisfying \eqref{eq_Zlim_s} and such that $\mathcal{Z}_0 = \widetilde{\mathcal{Z}}_0$ (we drop the superscript $\infty$ to ease the reading, with the understanding that, in this section, the process $\mathcal{Z}$ is not the one defined in \eqref{state at t}). We shall show that for every $t \ge 0$,
\begin{equation}
    \sup\limits_{s\le t} \sup\limits_{\| \phi \|_{\infty}\le 1} |\langle \mathcal{Z}_s - \widetilde{\mathcal{Z}}_s, \phi \rangle| = 0,
\end{equation}
where the supremum over $\| \phi \|_{\infty}\le 1$ is taken over the set of all bounded measurable functions on $\mathbb{R}^2$.

We start by bounding $\langle \mathcal{Z}_t,1 \rangle$ for every $t \ge 0$ and $\mathcal{Z}$ satisfying \eqref{eq_Zlim_s}. For $f \equiv 1$, we have for every $t \ge 0$
 \begin{align*}
     \langle \mathcal{Z}_t,1\rangle &=\langle\mathcal{Z}_0,1\rangle + \int_0^t\int_{\mathbb{R}^2} \big(b_1(x)+b_2(x)(t-s)-d(x,\mathcal{Z}_s)\big) \mathcal{Z}_s(dx) ds\\
     & \le \langle\mathcal{Z}_0,1\rangle + \int_0^t (B_1+B_2(t-s))\langle \mathcal{Z}_s,1\rangle ds.
 \end{align*}
Using Gronwall's Lemma, we obtain for every $t\geq 0$
\begin{equation}
\label{ineq_Zlim}
    \langle \mathcal{Z}_{t},1\rangle \le \langle \mathcal{Z}_{0},1\rangle e^{B_1t+\frac{B_2}{2}t^2}.
\end{equation}

Let $\varphi$ be a bounded measurable function on $\mathbb{R}^2$ and $t \ge 0$ and let us set for every $x \in \mathbb{R}^2$
\begin{equation}
\label{def_f}
\begin{aligned}
    f_s(x) &= \int_{\mathbb{R}^2} \varphi(y)g_{t-s}(x-y)dy, \text{ for every $s \in [0,t)$}\\
    f_t(x) &:= \varphi(x),
\end{aligned}
\end{equation}
where
\begin{equation}
\label{def_g}
    g_t(x) = \frac{1}{2 \pi \sigma^2 t} \exp\left(-\frac{|x|^2}{2 \sigma^2 t}\right).
\end{equation}
Since $\varphi$ is bounded, we have
\begin{equation*}
    \frac{\partial f_s}{\partial s} + \dfrac{\sigma^2}{2} \sum\limits_{i=1}^2 \frac{\partial^2 f_s}{\partial x_i^2} = \int_{\mathbb{R}^2} \varphi(y)\left[\frac{\partial g_{t-s}}{\partial s}(x-y) + \dfrac{\sigma^2}{2} \sum\limits_{i=1}^2 \frac{\partial^2 g_{t-s}}{\partial x_i^2}\right]dy = 0,
\end{equation*}
and
\begin{equation*}
    \lim\limits_{s \to t^{-}} f_s = \varphi = f_t.
\end{equation*}
Using \eqref{eq_Zlim_s} with $(f_s)_{0 \le s \le t}$, we have
\begin{equation}
\label{eq_phi}
\begin{aligned}
    \langle \mathcal{Z}_t,\varphi \rangle & = \langle \mathcal{Z}_0,f_0 \rangle + \int_0^t\int_{\mathbb{R}^2} \Bigg[\left(\sum\limits_{i=1}^2 \frac{\partial f_s}{\partial x_i}(x) \int_0^s\int_{\mathbb{R}^2} L^i_{s-r}(x-y) \mathcal{Z}_r(dy)dr  \right)\\
    & \hspace{3.5cm} + (b_1(x)+b_2(x)(t-s)-d(x,\mathcal{Z}_s))f_s(x) \Bigg]\mathcal{Z}_s(dx)ds,\\
    &= \langle \mathcal{Z}_0,f_0 \rangle + \int_0^t\int_{\mathbb{R}^2} \Bigg[\nabla_x f_s(x) \cdot \left(\int_0^s\int_{\mathbb{R}^2} L_{s-r}(x-y) \mathcal{Z}_r(dy)dr\right) \\
    & \hspace{3.5cm} + (b_1(x)+b_2(x)(t-s)-d(x,\mathcal{Z}_s)) f_s(x)\Bigg]\mathcal{Z}_s(dx)ds.
\end{aligned}
\end{equation}
We can then write, using the fact that $\mathcal{Z}_0 = \widetilde{\mathcal{Z}}_0$ by assumption,
\begin{align*}
    \langle \mathcal{Z}_t - \widetilde{\mathcal{Z}}_t,\varphi \rangle =& \int_0^t\int_{\mathbb{R}^2} \nabla_x f_s(x) \cdot \left(\int_0^s\int_{\mathbb{R}^2} L_{s-r}(x-y) (\mathcal{Z}_r-\widetilde{\mathcal{Z}}_r)(dy)dr\right)\mathcal{Z}_s(dx)ds \\
    & + \int_0^t\int_{\mathbb{R}^2} \nabla_x f_s(x) \cdot \left(\int_0^s\int_{\mathbb{R}^2} L_{s-r}(x-y) \widetilde{\mathcal{Z}}_r(dy)dr\right)(\mathcal{Z}_s-\widetilde{\mathcal{Z}}_s)(dx)ds\\
    & + \int_0^t\int_{\mathbb{R}^2}(b_1(x)+b_2(x)(t-s))f_s(x)(\mathcal{Z}_s - \widetilde{\mathcal{Z}}_s)(dx)ds\\
    & + \int_0^t\int_{\mathbb{R}^2} (d(x,\mathcal{Z}_s) - d(x,\widetilde{\mathcal{Z}}_s))f_s(x)\mathcal{Z}_s(dx)ds + \int_0^t\int_{\mathbb{R}^2} d(x,\widetilde{\mathcal{Z}}_s)f_s(x) (\mathcal{Z}_s-\widetilde{\mathcal{Z}}_s)(dx)ds.
\end{align*}
Let us bound each of these terms separately. We have
\begin{equation}
\label{ineq_4.3.1}
\begin{aligned}
    &\left| \int_0^t\int_{\mathbb{R}^2} \nabla_x f_s(x) \cdot \left(\int_0^s\int_{\mathbb{R}^2} L_{s-r}(x-y) (\mathcal{Z}_r-\widetilde{\mathcal{Z}}_r)(dy)dr\right)\mathcal{Z}_s(dx)ds \right| \\
    &\le \int_0^t \int_{\mathbb{R}^2}  |\nabla_x f_s(x)|\left| \int_0^s\int_{\mathbb{R}^2} L_{s-r}(x-y) (\mathcal{Z}_r-\widetilde{\mathcal{Z}}_r)(dy)dr \right| \mathcal{Z}_s(dx)ds
\end{aligned}
\end{equation}
Using Assumptions \eqref{L_1} and \eqref{L_2}, we have
\begin{align*}
    \left| \int_0^s\int_{\mathbb{R}^2} L_{s-r}(x-y) (\mathcal{Z}_r-\widetilde{\mathcal{Z}}_r)(dy)dr \right| & \le \int_0^s \|L_{s-r}\|_{\infty}\bigg| \left\langle \mathcal{Z}_r-\widetilde{\mathcal{Z}}_r, \frac{L_{s-r}}{\|L_{s-r}\|_{\infty}} \right\rangle \bigg| \\
    & \le  \int_0^s h_1(s-r) \sup\limits_{\|\phi\|_{\infty} \le 1} \langle \mathcal{Z}_r-\widetilde{\mathcal{Z}}_r, \phi \rangle dr\\
    & \le \|h_1\|_1 \ \sup\limits_{r \le s} \sup\limits_{\|\phi\|_{\infty} \le 1} \langle \mathcal{Z}_r-\widetilde{\mathcal{Z}}_r, \phi \rangle,
\end{align*}
where $\|h_1\|_1 = \displaystyle\int_0^{+\infty} |h_1(t)|dt<\infty$. Moreover, we have for every $s < t$
\begin{align*}
    \left| \nabla_x f_s(x) \right| &\le \|\varphi\|_{\infty} \int_{\mathbb{R}^2} |\nabla_x g_{t-s}(y)|dy \\
    & = \|\varphi\|_{\infty} \int_{\mathbb{R}^2} \frac{|y|}{2\pi \sigma^4 (t-s)^2} \exp\left(-\frac{|y|^2}{2\sigma^2(t-s)}\right)dy\\
    &= \|\varphi\|_{\infty}  \sqrt{\frac{\pi}{2\sigma^2(t-s)}}.
\end{align*}
Thus, combining \eqref{ineq_Zlim} and \eqref{ineq_4.3.1} we obtain
\begin{equation}
\label{ineq_Prop_4.3.1}
\begin{aligned}
        &\left| \int_0^t\int_{\mathbb{R}^2} \nabla_x f_s(x) \cdot \left(\int_0^s\int_{\mathbb{R}^2} L_{s-r}(x-y) (\mathcal{Z}_r-\widetilde{\mathcal{Z}}_r)(dy)dr\right)\mathcal{Z}_s(dx)ds \right| \\
        &\le \|\varphi\|_{\infty} \langle \mathcal{Z}_0, 1\rangle e^{B_1t + \frac{B_2}{2}t^2} \int_0^t \sqrt{\frac{\pi}{2\sigma^2(t-s)}} \sup\limits_{r \le s} \sup\limits_{\|\phi\|_{\infty} \le 1} \langle \mathcal{Z}_r-\widetilde{\mathcal{Z}}_r, \phi \rangle ds.
\end{aligned}
\end{equation}
Similarly, we have
\begin{align*}
    \left| \nabla_x f_s(x) \cdot \left(\int_0^s\int_{\mathbb{R}^2} L_{s-r}(x-y) \widetilde{\mathcal{Z}}_r(dy)dr\right) \right| &\le\|\varphi\|_{\infty} \sqrt{\frac{\pi}{2\sigma^2(t-s)}} \int_0^s h_1(s-r) \langle \widetilde{\mathcal{Z}}_r,1 \rangle dr \\
    &\le \|\varphi\|_{\infty} \|h_1\|_1 \langle \mathcal{Z}_0, 1\rangle e^{B_1t + \frac{B_2}{2}t^2} \sqrt{\frac{\pi}{2\sigma^2(t-s)}} ,
\end{align*}
and so
\begin{align}
\label{ineq_Prop_4.3.2}
    & \left|\int_0^t\int_{\mathbb{R}^2} \nabla_x f_s(x) \cdot \left(\int_0^s\int_{\mathbb{R}^2} L_{s-r}(x-y) \widetilde{\mathcal{Z}}_r(dy)dr\right)(\mathcal{Z}_s-\widetilde{\mathcal{Z}}_s)(dx)ds\right| \nonumber\\
    &\le \|\varphi\|_{\infty} \|h_1\|_1 \langle \mathcal{Z}_0, 1\rangle e^{B_1t + \frac{B_2}{2}t^2} \int_0^t  \sqrt{\frac{\pi}{2\sigma^2(t-s)}} \sup\limits_{\|\phi\|_{\infty} \le 1} \langle \mathcal{Z}_s-\widetilde{\mathcal{Z}}_s, \phi \rangle ds \nonumber \\
    &\le  \|\varphi\|_{\infty} \|h_1\|_1 \langle \mathcal{Z}_0, 1\rangle e^{B_1t + \frac{B_2}{2}t^2} \int_0^t \sqrt{\frac{\pi}{2\sigma^2(t-s)}} \sup\limits_{r \le s} \sup\limits_{\|\phi\|_{\infty} \le 1} \langle \mathcal{Z}_r-\widetilde{\mathcal{Z}}_r, \phi \rangle ds.
\end{align}
Next, we have
\begin{equation*}
    \left|\int_0^t\int_{\mathbb{R}^2}(b_1(x)+b_2(x)(t-s))f_s(x)(\mathcal{Z}_s - \widetilde{\mathcal{Z}}_s)(dx)ds\right| \le (B_1 + B_2 t) \int_0^t \|f_s\|_{\infty} \sup\limits_{r \le s} \sup\limits_{\|\phi\|_{\infty} \le 1} \langle \mathcal{Z}_r-\widetilde{\mathcal{Z}}_r, \phi \rangle ds,
\end{equation*}
and
\begin{equation*}
    \|f_s\|_\infty \le \|\varphi\|_{\infty} \int_{\mathbb{R}^2} g_{t-s}(x-y)dy = \|\varphi\|_{\infty}.
\end{equation*}
Hence,
\begin{equation}
\label{ineq_Prop_4.3.3}
    \left|\int_0^t\int_{\mathbb{R}^2}(b_1(x)+b_2(x)(t-s))f_s(x)(\mathcal{Z}_s - \widetilde{\mathcal{Z}}_s)(dx)ds\right| \le (B_1 + B_2 t) \|\varphi\|_{\infty} \int_0^t \sup\limits_{r \le s} \ \sup\limits_{\|\phi\|_{\infty} \le 1} \langle \mathcal{Z}_r-\widetilde{\mathcal{Z}}_r, \phi \rangle ds.
\end{equation}
Using \eqref{ineq_Zlim} and Assumption \eqref{d_2}, we have
\begin{equation}
\label{ineq_Prop_4.3.4}
\begin{aligned}
    \left|\int_0^t\int_{\mathbb{R}^2} \left(d(x,\mathcal{Z}_s) -d(x, \widetilde{\mathcal{Z}}_s)\right)f_s(x) \mathcal{Z}_s(dx)ds\right| \le \|\varphi\|_{\infty}& \int_0^t\int_{\mathbb{R}^2} |\langle \mathcal{Z}_s -\widetilde{\mathcal{Z}}_s, \Psi \rangle | \mathcal{Z}_s(dx)ds\\
    \le \|\varphi\|_{\infty}& \|\Psi\|_\infty \langle \mathcal{Z}_0,1 \rangle e^{B_1t + \frac{B_2}{2}t^2}\\
     & \times \int_0^t  \sup\limits_{r \le s} \ \sup\limits_{\|\phi\|_{\infty} \le 1} \langle \mathcal{Z}_r-\widetilde{\mathcal{Z}}_r, \phi \rangle ds.
\end{aligned}
\end{equation}
As stated in Remark~\ref{rem_assumptions}, Assumption \eqref{d_2} has not been used elsewhere, and is only used to obtain \eqref{ineq_Prop_4.3.4}. Finally, using Assumption \eqref{d_1} and \eqref{ineq_Zlim}, we have
\begin{equation}
\label{ineq_Prop_4.3.5}
    \left|\int_0^t\int_{\mathbb{R}^2} d(x,\mathcal{Z}_s)f_s(x) (\mathcal{Z}_s-\widetilde{\mathcal{Z}}_s)(dx)ds\right| \le C \langle \mathcal{Z}_0,1 \rangle e^{B_1t + \frac{B_2}{2}t^2} \|\varphi\|_\infty \int_0^t  \sup\limits_{r \le s} \ \sup\limits_{\|\phi\|_{\infty} \le 1} \langle \mathcal{Z}_r-\widetilde{\mathcal{Z}}_r, \phi \rangle ds.
\end{equation}
Thus, combining \eqref{ineq_Prop_4.3.1}, \eqref{ineq_Prop_4.3.2}, \eqref{ineq_Prop_4.3.3}, \eqref{ineq_Prop_4.3.4} and \eqref{ineq_Prop_4.3.5}, we obtain, for every bounded function $\varphi$ and $t \ge 0$
\begin{equation*}
     \langle \mathcal{Z}_t-\widetilde{\mathcal{Z}}_t, \varphi \rangle \le \widetilde{C} \langle \mathcal{Z}_0, 1\rangle e^{B_1t + \frac{B_2}{2}t^2} \|\varphi\|_{\infty} \int_0^t \left(\frac{1}{\sqrt{t-s}} + 1\right) \sup\limits_{r \le s} \ \sup\limits_{\|\phi\|_{\infty} \le 1} \langle \mathcal{Z}_r-\widetilde{\mathcal{Z}}_r, \phi \rangle ds,
\end{equation*}
for some constant $\widetilde{C}$ independent of $t$, so that for every $t \ge 0$
\begin{equation}\label{bound}
    \sup\limits_{\|\phi\|_{\infty} \le 1} \langle \mathcal{Z}_t-\widetilde{\mathcal{Z}}_t, \phi \rangle \le \widetilde{C} \langle \mathcal{Z}_0, 1\rangle e^{B_1t + \frac{B_2}{2}t^2} \int_0^t \left(\frac{1}{\sqrt{t-s}} + 1\right) \sup\limits_{r \le s} \ \sup\limits_{\|\phi\|_{\infty} \le 1} \langle \mathcal{Z}_r-\widetilde{\mathcal{Z}}_r, \phi \rangle ds.
\end{equation}
Let $t' \ge t \ge 0$. Changing variable in $s'=s+t-t'$, we have
\begin{align*}
& \int_0^{t'} \bigg(\frac{1}{\sqrt{t'-s}} + 1\bigg) \sup\limits_{r \le s} \ \sup\limits_{\|\phi\|_{\infty} \le 1} \langle \mathcal{Z}_r-\widetilde{\mathcal{Z}}_r, \phi \rangle ds \\
& = \int_{t-t'}^t \left(\frac{1}{\sqrt{t-s'}} + 1\right) \sup\limits_{r \le s'+(t'-t)} \ \sup\limits_{\|\phi\|_{\infty} \le 1} \langle \mathcal{Z}_r-\widetilde{\mathcal{Z}}_r, \phi \rangle ds'\\
& \geq \int_0^t \left(\frac{1}{\sqrt{t-s'}} + 1\right) \sup\limits_{r \le s'} \ \sup\limits_{\|\phi\|_{\infty} \le 1} \langle \mathcal{Z}_r-\widetilde{\mathcal{Z}}_r, \phi \rangle ds',
\end{align*}
from which we deduce that the integral appearing on the right hand side of~\eqref{bound} is a nondecreasing function of $t$. Thus, taking the supremum over $[0,t]$, we obtain for every $t \ge 0$
\begin{equation}
    \sup\limits_{s \le t} \ \sup\limits_{\|\phi\|_{\infty} \le 1} \langle \mathcal{Z}_s-\widetilde{\mathcal{Z}}_s, \phi \rangle \le \widetilde{C} \langle \mathcal{Z}_0, 1\rangle e^{B_1t + \frac{B_2}{2}t^2}  \int_0^t \left(\frac{1}{\sqrt{t-s}} + 1\right) \sup\limits_{r \le s} \ \sup\limits_{\|\phi\|_{\infty} \le 1} \langle \mathcal{Z}_r-\widetilde{\mathcal{Z}}_r, \phi \rangle ds.
\end{equation}
Let now $T \ge 0$. We have just proved that for every $0 \le t \le T$,
\begin{equation*}
    \sup\limits_{s \le t} \ \sup\limits_{\|\phi\|_{\infty} \le 1} \langle \mathcal{Z}_s-\widetilde{\mathcal{Z}}_s, \phi \rangle \le \widetilde{C} \langle \mathcal{Z}_0, 1\rangle e^{B_1T + \frac{B_2}{2}T^2}  \int_0^t \left(\frac{1}{\sqrt{t-s}} + 1\right) \sup\limits_{r \le s} \ \sup\limits_{\|\phi\|_{\infty} \le 1} \langle \mathcal{Z}_r-\widetilde{\mathcal{Z}}_r, \phi \rangle ds,
\end{equation*}
and, using \eqref{ineq_Zlim}, we also have that for every $0 \le t \le T$,
\begin{equation*}
    \sup\limits_{s \le t} \ \sup\limits_{\|\phi\|_{\infty} \le 1} \langle \mathcal{Z}_s-\widetilde{\mathcal{Z}}_s, \phi \rangle \le 2 \langle \mathcal{Z}_0, 1\rangle e^{B_1T+\frac{B_2}{2}T^2}.
\end{equation*}
As a consequence, we can use Theorem~3.3.1 in \cite{Amann} (a singular Gronwall lemma) to obtain that for every $T \ge 0$ and $0 \le t \le T$,
$$
\sup_{s \le t} \ \sup\limits_{\|\phi\|_{\infty} \le 1} \langle \mathcal{Z}_s-\widetilde{\mathcal{Z}}_s, \phi \rangle = 0,
$$
from which we conclude that $\mathcal{Z} = \widetilde{\mathcal{Z}}$. This was the result we were aiming for.

\subsection{Absolute continuity of the limiting measures}
\label{4.4}
In this section, we follow again the approach using the mild form proposed, for example, in \cite{Fournier} to show that, for every $t \ge 0$, the limiting measure $\mathcal{Z}^{\infty}_t$ is absolutely continuous with respect to Lebesgue measure on $\mathbb{R}^2$.

Let $\varphi$ be a bounded nonnegative measurable function and $t >0$, and recall the definition of $(f_s)_{s\le t}$ given in \eqref{def_f}. As in \eqref{eq_phi}, we have
\begin{align*}
    &\langle \mathcal{Z}_t^\infty,\varphi \rangle = \langle \mathcal{Z}_0^\infty,f_0 \rangle + \int_0^t\int_{\mathbb{R}^2} \Bigg[\nabla_x f_s(x) \cdot \left(\int_0^s\int_{\mathbb{R}^2} L_{s-r}(x-\Tilde{x}) \mathcal{Z}_r^\infty(d\Tilde{x})dr\right) \\
    & \hspace{5cm}+ (b_1(x)+b_2(x)(t-s)-d(x,\mathcal{Z}_s^\infty)) f_s(x)\Bigg]\mathcal{Z}_s^\infty(dx)ds.
\end{align*}
Using Fubini's theorem, we can write
\begin{align*}
    \langle \mathcal{Z}_0^\infty,f_0 \rangle & = \int_{\mathbb{R}^2}\int_{\mathbb{R}^2} \varphi(y)g_t(x-y)dy \mathcal{Z}_0^\infty(dx) =  \int_{\mathbb{R}^2} \varphi(y) \left(\int_{\mathbb{R}^2} g_t(x-y)\mathcal{Z}_0^\infty(dx)\right) dy.
\end{align*}
Similarly, we can write
\begin{align*}
    &\int_0^t\int_{\mathbb{R}^2} \nabla_x f_s(x) \cdot \left(\int_0^s\int_{\mathbb{R}^2} L_{s-r}(x-\Tilde{x}) \mathcal{Z}_r^\infty(d\Tilde{x})dr\right)\mathcal{Z}_s^\infty(dx)ds \\
    &= \int_0^t\int_{\mathbb{R}^2} \left(\int_{\mathbb{R}^2} \varphi(y) \nabla_x g_{t-s}(x-y)dy\right) \cdot \left(\int_0^s\int_{\mathbb{R}^2} L_{s-r}(x-\Tilde{x}) \mathcal{Z}_r^\infty(d\Tilde{x})dr\right)\mathcal{Z}_s^\infty(dx)ds\\
    &= \int_{\mathbb{R}^2} \varphi(y) \left[ \int_0^t \int_{\mathbb{R}^2} \nabla_x g_{t-s}(x-y) \cdot \left(\int_0^s\int_{\mathbb{R}^2} L_{s-r}(x-\Tilde{x}) \mathcal{Z}_r^\infty(d\Tilde{x})dr\right)\mathcal{Z}_s^\infty(dx)ds \right] dy.
\end{align*}
Using Assumption \eqref{L_1} and \eqref{ineq_Zlim}, we have
\begin{align*}
    \nabla_x g_{t-s}(x-y)  \cdot \left(\int_0^s\int_{\mathbb{R}^2} L_{s-r}(x-\Tilde{x}) \mathcal{Z}_r^\infty(d\Tilde{x})dr\right) & \le \left|\nabla_x g_{t-s}(x-y) \right| \times \left|\int_0^s\int_{\mathbb{R}^2} L_{s-r}(x-\Tilde{x}) \mathcal{Z}_r^\infty(d\Tilde{x})dr\right| \\
    & \le \left|\nabla_x g_{t-s}(x-y) \right| \int_0^s h_1(s-r) \langle \mathcal{Z}_r^\infty ,1\rangle dr\\
    & \le \int_0^s  h_1(s-r) dr \, \langle \mathcal{Z}_0^\infty,1 \rangle e^{B_1s+\frac{B_2}{2}s^2} \left|\nabla_x g_{t-s}(x-y) \right|.
\end{align*}
Thus,
\begin{align*}
     & \int_0^t \int_{\mathbb{R}^2} \nabla_x g_{t-s}(x-y) \cdot \left(\int_0^s\int_{\mathbb{R}^2} L_{s-r}(x-\Tilde{x}) \mathcal{Z}_r^\infty(d\Tilde{x})dr\right)\mathcal{Z}_s^\infty(dx)ds \\
    & \le C_t \langle \mathcal{Z}_0^\infty,1 \rangle \int_0^t \int_{\mathbb{R}^2} \left|\nabla_x g_{t-s}(x-y) \right| \mathcal{Z}_s^\infty(dx)ds.
\end{align*}
Again, we can write
\begin{align*}
    & \int_0^t\int_{\mathbb{R}^2} (b_1(x)+b_2(x)(t-s)-d(x,\mathcal{Z}^\infty_s) )f_s(x)\mathcal{Z}_s^\infty(dx)ds \\
    & \le (B_1+B_2t) \int_0^t\int_{\mathbb{R}^2} \left(\int_{\mathbb{R}^2} \varphi(y) g_{t-s}(x-y) dy \right) \mathcal{Z}_s^\infty(dx)ds \\
    & = (B_1+B_2t) \int_{\mathbb{R}^2} \varphi(y) \left( \int_0^t\int_{\mathbb{R}^2}g_{t-s}(x-y) \mathcal{Z}_s^\infty(dx)ds \right) dy.
\end{align*}
Thus,
\begin{align*}
    \langle \mathcal{Z}_t^\infty,\varphi \rangle & \le \int_{\mathbb{R}^2} \varphi(y) \left[ \int_{\mathbb{R}^2} g_t(x-y)\mathcal{Z}_0^\infty(dx) +  C_t \int_0^t \int_{\mathbb{R}^2} \left( \left|\nabla_x g_{t-s}(x-y) \right| + g_{t-s}(x-y) \right) \mathcal{Z}_s^\infty(dx)ds \right]dy \\
    & \le  \int_{\mathbb{R}^2} \varphi(y) H(t,y) dy,
\end{align*}
where
\begin{equation*}
    H(t,y) = \int_{\mathbb{R}^2} g_t(x-y)\mathcal{Z}_0^\infty(dx) +  C_t \int_0^t \int_{\mathbb{R}^2} \left( \left|\nabla_x g_{t-s}(x-y) \right| + g_{t-s}(x-y) \right) \mathcal{Z}_s^\infty(dx)ds,
\end{equation*}
is nonnegative. Now,
\begin{equation*}
    \int_{\mathbb{R}^2} \int_{\mathbb{R}^2} g_t(x-y)\mathcal{Z}_0^\infty(dx)dy = \int_{\mathbb{R}^2} \left(\int_{\mathbb{R}^2} g_t(x-y)dy\right)\mathcal{Z}_0^\infty(dx) = \langle \mathcal{Z}_0^\infty,1 \rangle,
\end{equation*}
and, using \eqref{ineq_Zlim},
\begin{align*}
& \int_{\mathbb{R}^2} C_t \int_0^t \int_{\mathbb{R}^2} \left( \left|\nabla_x g_{t-s}(x-y) \right| + g_{t-s}(x-y) \right) \mathcal{Z}_s^\infty(dx)ds dy \\
&= C_t \int_0^t \int_{\mathbb{R}^2} \left( \int_{\mathbb{R}^2} \left|\nabla_x g_{t-s}(x-y) \right| + g_{t-s}(x-y) dy \right) \mathcal{Z}_s^\infty(dx)ds\\
&= C_t \int_0^t \left(\sqrt{\frac{\pi}{2(t-s)}} + 1 \right)   \langle \mathcal{Z}_s^\infty, 1 \rangle ds\\
&\le  C_t \langle \mathcal{Z}_0^\infty, 1\rangle e^{B_1t+\frac{B_2}{2}t^2} \int_0^t \left(\sqrt{\frac{\pi}{2s}} + 1 \right) ds < \infty.
\end{align*}
Thus $H(t,y)$ is integrable and therefore, since the inequality holds for every bounded function $\varphi$, we can conclude that for every $t > 0$, $\mathcal{Z}_t^\infty$ is absolutely continuous with respect to Lebesgue measure.

The derivation of the equation for the density is classical, and is therefore left aside.

\medskip
\begin{acknowledgements}
The author would like to thank the members of the NEMATIC consortium, funded by the ANR grant ANR21-CE40010-01, for interesting and helpful discussions on the modelling of filamentous fungi and experimental data. She particularly thanks Milica Toma\v{s}evi\'{c} and Amandine V\'{e}ber for their guidance on this project. This work was partly funded by the Chair Programme \emph{Mathematical Modelling and Biodiversity} (Ecole Polytechnique, Museum National d'Histoire Naturelle, Veolia Environnement, Fondation X).
\end{acknowledgements}

\bibliographystyle{plain}
\bibliography{bibliography.bib}

\end{document}